\documentclass[11pt]{article}
\usepackage{lineno,hyperref,graphicx,amsmath,amsfonts,enumerate,amsthm,appendix,caption,lipsum}
\usepackage{color}
\usepackage[utf8]{inputenc}
\usepackage[english]{babel}
\renewcommand{\vec}[1]{\boldsymbol{#1}}
\newcommand{\drop}[1]{}
\newcommand{\norm}[1]{\left\lVert#1\right\rVert}
\modulolinenumbers[5]

\newtheorem{proposition}{Proposition}
\theoremstyle{definition}
\newtheorem{definition}{Definition}
\theoremstyle{remark}
\newtheorem{remark}{Remark}
\newcommand\blfootnote[1]{%
  \begingroup
  \renewcommand\thefootnote{}\footnote{#1}%
  \addtocounter{footnote}{-1}%
  \endgroup
}
\begin{document}
\title{Basis functions for residual stresses\blfootnote{This article has been published in `Applied Mathematics and Computation' and can be accessed here: \href{https://authors.elsevier.com/a/1bMDlLvMgOYoK}{https://authors.elsevier.com/a/1bMDlLvMgOYoK}}}

%% Group authors per affiliation:
\author{Sankalp Tiwari\thanks{snklptwr@gmail.com, sankalpt@iitk.ac.in} \and Anindya Chatterjee\thanks{anindya100@gmail.com, anindya@iitk.ac.in}}

\date{Mechanical Engineering, IIT Kanpur}
\maketitle

\begin{abstract}
We consider arbitrary preexisting residual stress states in arbitrarily shaped, unloaded bodies. These stresses must be self-equilibrating and traction free. Common treatments of the topic tend to focus on either the mechanical origins of the stress, or methods of stress measurement at certain locations. Here we take the stress field as given and consider the problem of approximating any such stress field, in a given body, as a linear combination
of predetermined fields which can serve as a basis. We consider planar stress states in detail, and introduce an extremization problem that leads to a linear eigenvalue problem. Eigenfunctions of that problem form an orthonormal basis for all possible residual stress states of sufficient smoothness. In numerical examples, convergence of the approximating stress fields is demonstrated in the $L^2$ norm for continuous stress fields as well as for a stress field with a simple discontinuity. Finally, we outline the extension of our theory to three dimensional bodies and states of stress. Our approach can be used to describe arbitrary preexisting residual stress states in arbitrarily shaped bodies using basis functions that are determined by the body geometry alone.
\end{abstract}
%\linenumbers

\section{Introduction}
\label{intro}
We consider basis functions for interpolating residual stress fields in finite bodies.
In particular, we consider bodies that are arbitrarily shaped, not subjected to body forces, in equilibrium, and with traction free boundaries, but with nonzero internal residual stresses.
The physical sources of the residual stresses may be prior manufacturing processes, deformation history, thermal gradients, or other phenomena. Here we are interested solely in mathematical ways to discuss or describe residual stress fields that already exist, independent of the physical mechanisms that have produced them.

For example, if residual stress states are experimentally determined at $N$ points on a manufactured component, and if reasonable smoothness in residual stress variations can be assumed, how should the residual stresses be interpolated between those points in space? As another example, in a metal forming simulation, can final residual stresses in the formed component be reported using some sequence of orthogonal basis functions that is specifically constructed, {\em in advance}, for the domain of interest?

With the above motivation, we seek self-equilibrating traction-free fields $\vec{\phi_i}$ defined on the finite body of interest, such that linear combinations $\displaystyle \sum_{i=1}^{\infty} a_i \vec{\phi_i}$ can capture any sufficiently regular residual stress field.   

In this paper, we will construct such fields $\vec{\phi_i}$ {\em via} stationary values of a suitable quadratic functional. These fields $\vec{\phi_i}$ will serve as a basis for representing arbitrary residual stress fields in bodies of a given but arbitrary shape, without regard for the physical source of the residual stress. To the best of our knowledge, such a basis has not been presented in the mechanics literature before. The construction of such a basis is not obvious in advance. For example, readers familiar with vibration theory \cite{rayleigh} may be interested to see that the stress fields induced by vibration modes {\em cannot} be used for such $\vec{\phi_i}$, because those modal stresses necessarily satisfy the strain-compatibility conditions of linear elasticity while not satisfying equilibrium, whereas residual stresses necessarily satisfy equilibrium and violate strain-compatibility equations of linear elasticity\footnote{%
	Equilibrium, zero tractions {\em and} compatibility lead to zero stresses as a unique solution.}. To see the latter easily, we can use the result that for a linearly elastic  body subjected to given tractions and body forces, the displacement is unique up to a rigid motion (see page 45, theorem 4.3.1 of \cite{knops}). The solution to zero traction and zero body force is therefore zero stress and zero displacement, unique up to rigid body motions, by the above result. The stress corresponding to rigid body motions is zero. Hence, non-zero residual stresses cannot be caused by compatible strains in linear elasticity.

As motivation for the development that is to follow, in order to demonstrate that vibration mode-induced (or modal) stresses {\em cannot} be used to construct a basis for residual stresses, we choose a candidate residual stress field in an annular domain of inner radius $0.1$ and outer radius $0.3$, with components 
\begin{equation}
\label{demo}
\begin{split}
\sigma_{rr}(r,\theta)=\left(-\frac{0.067}{r^2}+\frac{1.6}{r}-12.833+40 r-41.667 r^2\right) \cos{3\theta},\\
\sigma_{r\theta}(r,\theta)= \left(-\frac{0.022}{r^2} +5.5 -40 r + 75  r^2\right) \sin{3\theta},\\
\sigma_{\theta\theta}(r,\theta)= \left(3.667-40 r+100r^2\right)\cos{3\theta}.
\end{split}
\end{equation}
We will properly motivate and use this stress field later in the paper, after presenting our theory. Here we merely attempt to {\em numerically} approximate the above stress field with the first $N$ modal stresses on this domain, with $1\leq N \leq 50.$ An approximation error $E_N$ (which will be described fully in due course) is plotted against $N$ in figure \ref{En_fvm}. We see that the approximation does not seem to be converging. The implications of figure \ref{En_fvm}, which is given here only for motivation, will be clearer as we present our theory in subsequent sections.
\begin{figure}[h!]
	\centering
	\includegraphics [width=0.5\textwidth]{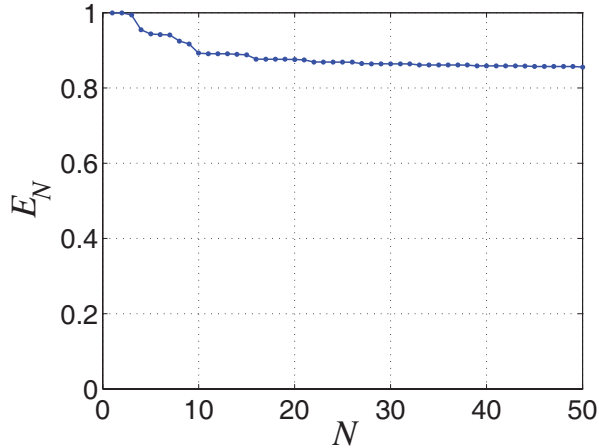}
	%	\vspace{-30mm}
	\caption{Approximation error versus number of (vibrational) modal stress fields used. Convergence to zero is not apparent and seems unlikely.}	
	\label{En_fvm}
\end{figure}
Readers may note that residual stresses in a component can be either beneficial or harmful, depending on the application. For example, they can impede the growth of surface microcracks and extend fatigue life, or cause warping in manufactured components, respectively. In either case, it is important to characterize a body's residual stress state with sufficient accuracy, both in the bulk and at the surface. Readers may refer to \cite{schajer2013practical,withers2001residual,withers2001residual2} for comprehensive discussions on the origin and measurement of residual stresses from differing sources and at different length scales. Broadly, some common sources of residual stresses are thermal effects \cite{boley,eslami}, inclusions and defects \cite{eshelby1957determination,eshelby1959elastic,eshelby1957elastic,eshelby1966simple,kroener,mura}, and biological growth \cite{goriely,zurlo,swain,epstein}, in addition to prior inelastic deformation. 

A substantial amount of literature on residual stresses pays explicit attention to incompatibility, e.g., through equations of the form $\nabla^4 \phi = \eta$, where nonzero $\eta$ is the source of incompatibility \cite{kroener,mura,zurlo}.
As mentioned above, we directly seek a basis for expanding and interpolating the stress components without approaching the problem through specific choices of $\eta$, i.e., through specific sources or types of incompatibility. We acknowledge here the work of Hoger \cite{hoger1986,hoger1985}, who discussed the general residual stress fields possible in an elastic cylinder, but did not seek to develop a basis for interpolation on arbitrary geometries as we do here. Her papers led to interesting subsequent work on elastic bodies with residual stress, in which the strain energy density is considered to be a function of both the deformation gradient tensor and the initial residual stress. These works, like ours, make no assumptions about the origin of the prescribed residual stress \cite{hoger1986,hoger1985,shams,gower,ISS,shariff,hartig,destrade}. These works, especially those concerned with calculating the optimal residual stress that results in a targeted Cauchy stress (e.g. \cite{ciarletta}), deviate almost immediately from our approach in that they focus on elastic bodies. 

We also distinguish our approach from a more restricted interpolation employed in some destructive measurement techniques for residual stresses. In those techniques, elasticity-based relationships between the measured strain data  \cite{prime1999residual,schajer2007residual,akbari,beghini}, and the tractions that {\em were} acting on surfaces that have since been exposed by cutting \cite{ballard}, are the key considerations. In such measurements, the stress is often interpolated along a single spatial coordinate (like depth of cut), using splines, polynomials, Fourier series, etc. Unlike those interpolants defined on specific line segments, here we will develop self-equilibrating, traction free, tensor valued interpolants for the entire body without appeal to any underlying material constitutive relations. We also acknowledge the challenging problem of inversion of boundary data (displacements, strains) to estimate the residual stress in a three dimensional body \cite{ballard,robertson,acoustic,gao}. The orthonormal basis we develop here, in such applications, may ease the need for statistical regularization \cite{schajerinv,faghidian}. Such potential applications provide yet another motivation for our work.

Finally we distinguish our approach from stress-based formulations derived in linear elasticity using variations of a positive definite functional of the stress gradient \cite{pobedrja1978,pobedrja1980,li,markenscoff}, an example of which is the Beltrami-Michell equation \cite{gurtin}. These formulations, too, refer specifically to linearly elastic materials, and do {\em not} construct basis functions. Our aims are quite different, as explained above. In particular, we will consider variations of a functional involving the stress gradient, which leads to an  eigenvalue problem, which in turn yields a basis we can use.
In the applied mathematics literature, there are similar issues studied using the somewhat simpler Stokes operator from incompressible fluid mechanics (see e.g., \cite{temam}; we will discuss these similarities briefly near the end of the paper). Readers wishing to read a general discussion of the spectral theorem may see, e.g., \cite{halmos}. However, our discussion is less formal, accessible to a broader audience, and resembles the 
development of classical vibration theory \cite{rayleigh}.   

Our basic formulation, though first developed below for two dimensions, is extended to three dimensions at the end of the paper.

We close this introduction with a brief description of the notation used in this paper. The dot product `$\cdot$' between two tensors of the same order represents total tensor contraction. Using Einstein's summation convention,
\[
\vec{A}\cdot \vec{B}= 
\begin{cases}
A_i B_i& \mbox{if } \,\, \vec{A} \,\, \mbox{and} \,\, \vec{B} \,\, \mbox{are vectors}, \\
A_{ij} B_{ij}& \mbox{if } \,\, \vec{A} \,\, \mbox{and} \,\, \vec{B} \,\, \mbox{are second order tensors}, \\
A_{ijk} B_{ijk}& \mbox{if } \,\, \vec{A} \,\, \mbox{and} \,\, \vec{B} \,\, \mbox{are third order tensors},
\end{cases}
\]
where $A_i,A_{ij},A_{ijk}$ etc.\ are the Cartesian components of the tensor $\vec{A}$ (likewise for $\vec{B}$).
For a second order tensor $\vec{A}$, $\mbox{div} \vec{A}$  represents $A_{ij,j} \vec{e}_i $, where a subscript following a comma denotes a partial derivative. For a vector $\vec{v}$, $\vec{A} \vec{v}$ represents $A_{ij}v_j \vec{e}_i$.
The dyadic product $\vec{u} \otimes \vec{v}$ for vectors $\vec{u}$ and $\vec{v}$ is defined by its action on a vector $\vec{w}$ as $(\vec{u} \otimes \vec{v})\vec{w}=(\vec{v}\cdot \vec{w})\vec{u}$.

\section{Problem statement} 
\label{PSsec}
Let $\Omega$ be an open, bounded, sufficiently regular domain in $\mathbb{R}^d$, with $d=2$ for the moment (the extension to $d=3$ is discussed at the end of the paper), with area $|\Omega|$. The unit outward normal $\vec{n}$ at each point on the boundary $\partial \Omega$ is assumed well defined\footnote{%
	Isolated corners can be rounded out using tiny radii, for simplicity. In finite element approximations, the weak formulation allows a piecewise $C^1$ boundary.}.

Let us denote the set of symmetric second order tensor fields by ``Sym''. We define:
\begin{equation}
\label{eqdefS}
\mathcal{S}=\biggl\{\vec{\sigma} \left | \vec{\sigma} \in \text{Sym}, \hspace{1mm} \mbox{div} \, \vec{\sigma} = \vec{0}, \hspace{1mm} \vec{\sigma} \vec{n}=\vec{0}, \hspace{1mm} \int_{\Omega} \vec{\sigma} \cdot \vec{\sigma} \, dA < \infty,\hspace{1mm} \int_{\Omega} \nabla \vec{\sigma} \cdot \nabla \vec{\sigma} \, dA < \infty \right. \biggr\},
\end{equation}
where the five conditions included imply symmetry, equilibrium, zero tractions, square integrability of stresses, and square integrability of stress gradients respectively; and $dA$ is an infinitesimal area element of the domain $\Omega$.
The norm of any $\vec{\sigma} \in \mathcal{{S}}$ is taken to be
\begin{equation}
\label{norm}
\norm{\vec{\sigma}} = \left ( \int_{\Omega} \vec{\sigma} \cdot \vec{\sigma} \, dA \right )^{\frac{1}{2}}.
\end{equation}
Let $\bar{\mathcal{S}}$ be the closure of $\mathcal{S}$. All residual stress fields of interest to us are elements of $\bar{\mathcal{S}}$. 
We seek a sequence of fields $\vec{\phi_i}$ that span $\mathcal{\bar{S}}$.

For orthogonality conditions discussed later in the paper, we use the inner product between two elements $\vec{\sigma_1}$ and $\vec{\sigma_2}$ of $\bar{\mathcal{S}}$ as follows
$$ (\vec{\sigma_1},\vec{\sigma_2})=\int_{\Omega} \vec{\sigma_1} \cdot \vec{\sigma_2} \, dA.$$

\section{Solution approach via an extremization problem}
\label{NZ}
Let us seek stationary points of the functional
\begin{equation}
\label{extreme}
J_0(\vec{\widetilde{\sigma}})=\frac{1}{2}\int_{\Omega} \nabla \vec{\widetilde{\sigma}} \cdot \nabla \vec{\widetilde{\sigma}} \, dA,
\end{equation}
over sufficiently regular\footnote{%
	For the calculus of variations, we will assume continuous second partial derivatives. In
	finite element approximations, the weak solution requires lower smoothness.}
$\vec{\widetilde{\sigma}}$ in $\mathcal{S}$, subject to the normalization constraint \\$\displaystyle \int_{\Omega} \vec{\widetilde{\sigma}} \cdot \vec{\widetilde{\sigma}} \, dA =1$. 

We note that for any nonzero residual stress field, the quantity $J_0$ must be nonzero (see e.g., \cite{hoger1986}).

We will use the calculus of variations \cite{Courant}. Since the constraint $\mbox{div} \, \vec{\widetilde{\sigma}} = \vec{0}$ is defined pointwise in space, we introduce a spatially varying Lagrange multiplier $\vec{\mu}$ for it. Since $\displaystyle \int_{\Omega} \vec{\widetilde{\sigma}} \cdot \vec{\widetilde{\sigma}} \, dA =1$ is a scalar integral constraint, we use a scalar Lagrange multiplier $\displaystyle \frac{\lambda}{2}$ for it.
We then consider variations of
\begin{equation}
\label{e1}
J(\vec{\hat{\sigma}})=\int_{\Omega} \left \{ \frac{1}{2} \, \nabla \vec{\hat{\sigma}} \cdot \nabla \vec{\hat{\sigma}}  - \frac{\lambda}{2} \, \left(\vec{\hat{\sigma}} \cdot \vec{\hat{\sigma}} - \frac{1}{|\Omega|}\right) - \vec{\mu} \cdot (\mbox{div} \, \vec{\hat{\sigma}})
\right \} \, dA,
\end{equation}
where we have used a ``hat'' instead of a ``tilde'' on $\vec{\hat{\sigma}}$ because it belongs to the larger, or less restricted, set
\begin{equation}
\nonumber
\mathcal{R}=\biggl\{\vec{\hat{\sigma}} \left | \vec{\hat{\sigma}} \in \text{Sym}, \hspace{1mm} \vec{\hat{\sigma}} \vec{n}=\vec{0},  \int_{\Omega} \vec{\hat{\sigma}} \cdot \vec{\hat{\sigma}} \, dA < \infty, \int_{\Omega} \nabla \vec{\hat{\sigma}} \cdot \nabla \vec{\hat{\sigma}} \, dA < \infty \right. \biggr\}.
\end{equation}

If a stationary point of Eq.\ \ref{e1} is $\vec{\sigma}$ then, for arbitrary infinitesimal variations $\vec{\zeta} \in \mathcal{R}$, we must have
\begin{equation}
\nonumber
\int_{\Omega} \left\{ \nabla \vec{\sigma} \cdot \nabla \vec{\zeta}  - \lambda \vec{\sigma} \cdot \vec{\zeta} - \vec{\mu} \cdot (\mbox{div} \,\vec{\zeta}) \right \} \, dA = 0.
\end{equation}
Using integration by parts and the divergence theorem, we obtain
\begin{equation}
\label{ev1}
\int_{\partial \Omega} \left \{(\nabla \vec{\sigma} \circ \vec{\zeta}) \cdot \vec{n} - \vec{\mu} \cdot (\vec{\zeta} \vec{n}) \right \} ds - \int_{\Omega} \left \{ \Delta \vec{\sigma} - \nabla \vec{\mu} + \lambda \vec{\sigma} \right \} \cdot \vec{\zeta} \, dA = 0,
\end{equation}
where $\vec{A} \circ \vec{B}=A_{ijk} B_{ij} \vec{e}_k$ in Cartesian coordinates for a third order tensor $\vec{A}$, second order tensor $\vec{B}$, and unit vectors
$\vec{e}_k$.

In Eq.\ \ref{ev1}, since $\vec{\zeta} \in \mathcal{R}$, $\vec{\zeta} \vec{n}$ on $\partial \Omega$ is zero, yielding
\begin{equation}
\label{two_terms}
\int_{\partial \Omega} (\nabla \vec{\sigma} \circ \vec{\zeta}) \cdot \vec{n} \, ds - \int_{\Omega} \left \{ \Delta \vec{\sigma} - \nabla \vec{\mu} + \lambda \vec{\sigma} \right \} \cdot \vec{\zeta} \, dA = 0.
\end{equation}
By considering the set of $\vec{\zeta}$ which are zero on $\partial \Omega$, we conclude that\footnote{%
	Since $\vec{\zeta}$ is symmetric, by localizing it near any $\vec{x} \in \Omega$
	we conclude that the integrand at $\vec{x}$ is skew symmetric.}
$$ -\Delta \vec{\sigma} + \nabla \vec{\mu} - \lambda \vec{\sigma} = \vec{R}  \hspace{5mm} \text{in} \hspace{1mm} \Omega,$$
where $\vec{R}$ is some skew symmetric second order tensor field; and where the scalar eigenvalue $\lambda$ and the vector field $\vec{\mu}$ need to be determined along with $\vec{\sigma}$. Adding the above equation to its transpose and dividing by two,
\begin{equation}
\label{q1q2} -\Delta \vec{\sigma} + \nabla_s \vec{\mu} - \lambda \vec{\sigma} = \vec{0} \hspace{5mm} \text{in} \hspace{1mm} \Omega,
\end{equation}
where
\begin{equation}
\nonumber
\nabla_s \vec{\mu} = \frac{\nabla \vec{\mu} + \left ( \nabla \vec{\mu} \right )^T}{2}.
\end{equation}
Equation \ref{two_terms} reduces to the surface integral alone, i.e.,
\begin{equation}
\label{natural}
\int_{\partial \Omega} (\nabla \vec{\sigma} \circ \vec{\zeta}) \cdot \vec{n} \, ds = 0.
\end{equation}

Considering $\vec{\zeta}$ on the boundary, at each point we have $\vec{\zeta n}=\vec{0}$, so $\vec{n}$ is an eigenvector of $\vec{\zeta}$. Since $\vec{\zeta}$ is symmetric, the local tangent vector $\vec{t}$ must be the other eigenvector (we are in two dimensions).
It follows that we can consider $\vec{\zeta} = \kappa(s) \vec{t} \otimes \vec{t}$ for any scalar $\kappa(s)$ varying arbitrarily along the boundary.
The arbitrariness of $\kappa(s)$ implies that
\begin{equation} \label{q2q3} ( \nabla \vec{\sigma} \circ \{ \vec{t} \otimes \vec{t} \} ) \cdot \vec{n} = 0 \end{equation}
everywhere on the boundary $\partial \Omega$.
Using indicial notation,
\begin{equation}
\label{q3q4}
( \nabla \vec{\sigma} \circ \{ \vec{t} \otimes \vec{t} \} ) \cdot \vec{n} = \sigma_{ij,k} n_k t_i t_j= \nabla_n \vec{\sigma} \cdot (\vec{t} \otimes \vec{t}) = 0
\end{equation}
everywhere on the boundary, where $\nabla_n$ denotes the derivative in the locally normal direction.
Less formally, the normal gradient of the circumferential tensile stress is zero at the boundary. If the domain is circular, this circumferential stress is the hoop stress.

Finally, variation of the Lagrange multiplier $\vec{\mu}$ gives the equilibrium condition
$$ \mbox{div} \, \vec{\sigma}=\vec{0},$$
and variation of the Lagrange multiplier $\displaystyle \frac{\lambda}{2}$ gives
$$ \int_{\Omega} \vec{\sigma} \cdot \vec{\sigma} \, dA = 1.$$

To summarize, any sufficiently regular unit-norm stationary point of $J_0$ in $\mathcal{{S}}$, assuming for simplicity that one exists, is a solution to 
the following eigenvalue problem:
\begin{equation}
\label{three_eqns}
\begin{array}{cccl}
-\Delta \vec{\sigma} + \nabla_s \vec{\mu}  = \lambda \vec{\sigma} & \text{ and } & \mbox{div} \, \vec{\sigma} = \vec{0} & \text{ in} \hspace{1mm} \Omega,\\
\vec{\sigma n} = \vec{0} & \text{ and } & \nabla_n \vec{\sigma} \cdot (\vec{t} \otimes \vec{t}) = 0 & \text{ on} \hspace{1mm} \partial \Omega.
\end{array}
\end{equation}
This eigenvalue problem can be solved on arbitrary domains using the finite element method, and we will present some such solutions later in this paper. For the simple case of an annular domain, it can also be solved as a two-point boundary value problem using ODE solvers after separation of variables, and we will present such solutions as well, obtaining complete agreement with finite element solutions.

Proceeding now with our theoretical development, our primary claim is that the sequence of eigenfunctions $\vec{\sigma}_k$, computed for a given domain $\Omega$, forms a basis for $\bar{\mathcal{S}}$ defined on $\Omega$. Any state of residual stress in $\bar{\mathcal{S}}$ can be expressed as a linear combination of these basis functions. We shall henceforth denote these stress-eigenfunctions as $\vec{\phi}$.
%%%%%%%%%%%%%%%%%%%%%%%%%%%%%%%%%%%%%%%%%%%%%%%%%%%%%%%%

\section{Orthonormality of the eigenfunctions}
Let $\lambda$ be an eigenvalue, and $\vec{\phi}$ and $\vec{\mu}$ represent the corresponding eigenfunction. Let $\vec{\sigma}$ be any element of ${\mathcal{S}}$
(recall Eq.\ \ref{eqdefS}).
Consider the inner product of the first equation in \ref{three_eqns}
with $\vec{\sigma}$, i.e.,
\begin{equation}
\label{ip}
\int_{\Omega} \left ( - \Delta \vec{\phi} + \nabla_s \vec{\mu}  - \lambda \vec{\phi} \right )  \cdot \vec{\sigma} \, dA = 0,
\end{equation} 
which reduces to (see appendix \ref{math_reductions})

\begin{equation}
\label{ip2} \int_{\Omega} \left ( \nabla \vec{\phi} \cdot \nabla \vec{\sigma}  - \lambda \vec{\phi} \cdot \vec{\sigma} \right )  \, dA = 0
\end{equation}
for any eigenvalue-eigenfunction pair $(\lambda, \vec{\phi})$ and any $\vec{\sigma} \in {\mathcal S}$.

Now let $(\lambda_p, \vec{\phi}_p, \vec{\mu}_p)$ and $(\lambda_q, \vec{\phi}_q, \vec{\mu}_q)$ be two distinct eigenvalue-eigenvector sets of Eq.\ \ref{three_eqns}.
By Eq.\ \ref{ip2},
\begin{equation}
\label{Rp1}
\begin{split}
\int_{\Omega} \nabla \vec{\phi}_p \cdot \nabla \vec{\phi}_q \, dA = \lambda_p \int_{\Omega} \vec{\phi}_p \cdot \vec{\phi}_q \, dA, \\
\int_{\Omega} \nabla \vec{\phi}_p \cdot \nabla \vec{\phi}_q \, dA = \lambda_q \int_{\Omega} \vec{\phi}_p \cdot \vec{\phi}_q \, dA,
\end{split}
\end{equation}
and
if $\lambda_p \neq \lambda_q$, then
\begin{equation}
\label{ort1}
\int_{\Omega} \vec{\phi}_p \cdot \vec{\phi}_q \, dA =0 \mbox{ and } \int_{\Omega} \nabla \vec{\phi}_p \cdot \nabla \vec{\phi}_q \, dA =0.
\end{equation}
If $\lambda_p = \lambda_q$ but $\vec{\phi}_p \neq \vec{\phi}_q$, then we can choose
$\vec{\phi}_p$ and $\vec{\phi}_q$ to be orthogonal, and Eq.\ \ref{ort1} still holds.
Finally, if $\lambda_p = \lambda_q$ and 
$\vec{\phi}_p = \vec{\phi}_q$ but $\vec{\mu}_p \neq \vec{\mu}_q$, then $\nabla_s \vec{\mu}_p = \nabla_s \vec{\mu}_q$, and there is no distinction between these two cases.

Following arguments used by \cite{rayleigh}, we note that the eigenvalues $\lambda$ are real and positive. To obtain a contradiction, if $\lambda_p$ is complex with corresponding complex eigenfunction $\vec{\phi}_p$, then by the linearity of
Eq.\ \ref{three_eqns} it follows that their complex conjugates $\overline{\lambda_p} = \lambda_q$ and $\overline{\vec{\phi}_p} = \vec{\phi}_q$ give another solution pair. Using these two eigenfunctions in either of Eqs.\ \ref{Rp1}, we obtain a contradiction; so $\lambda$ is real. The eigenfunctions are real as well. Next, using the same $\vec{\phi}$ twice (i.e., $p=q$), we conclude that $\lambda > 0$ because the left hand side is strictly positive for any nonzero residual stress.  

We thus have an orthogonal sequence of eigenfunctions, satisfying Eq.\ \ref{ort1} whenever $p \neq q$.
The orthogonal sequence of stress eigenfunctions $\vec{\phi}_p$ is assumed to be normalized such that
$$\int_{\Omega} \vec{\phi}_p \cdot \vec{\phi}_p \, dA =1, \quad p = 1, 2, 3, \cdots$$
to obtain an {\em orthonormal} sequence, with
$$\int_{\Omega} \nabla \vec{\phi}_p \cdot \nabla \vec{\phi}_p \, dA = \lambda_p.$$
We can arrange this sequence\footnote{%
	\label{ftnt1} In some cases we may restrict attention to a subset of eigenfunctions. For an annular domain, for example, we may sometimes consider only
	eigenfunctions with a fixed circumferential
	wave number (e.g., $m=3$).}
simply in order of increasing $\lambda_p$.
%%%%%%%%%%%%%%%%%%%%%%%%%%%%%%%%%%%%%%%%

\section{Basis of $\mathcal{\bar{S}}$} \label{basis_sec} 
Consider the sequence $(\lambda_p, \vec{\phi}_p, \vec{\mu}_p)$, $p = 1, 2, \cdots$.
There are infinitely many such eigenvalue-eigenfunction pairs, i.e., the sequence is not finite. For proof, we argue by contradiction.

Assume that only a finite number $N$ of such eigenvalue-eigenfunction pairs exist.

Let ${\mathcal S}_N$ be the subspace of $\mathcal{{S}}$ spanned by the finite sequence $\left \{ \vec{\phi}_p \right \}$, $p = 1, 2, \cdots, N$. Let
${\mathcal S}_{N\perp}$ be the orthogonal complement of ${\mathcal S}_N$ in ${\mathcal S}$. Let us now
extremize $J_0$ (recall Eq.\ \ref{extreme}) within ${\mathcal S}_{N\perp}$. To the extremizer $\vec{\sigma}$, restriction to ${\mathcal S}_{N\perp}$ adds $N$ integral constraints to the previous extremization problem, namely
\begin{equation}
\label{ortnew}
\int_{\Omega} \vec{\phi}_p \cdot \vec{\sigma} \, dA = 0, \quad p = 1, 2, \cdots, N,
\end{equation}
for which we introduce $N$ new scalar Lagrange multipliers, $\nu_1, \nu_2, \cdots, \nu_N$,
and obtain the new equations (recall Eq.\ \ref{three_eqns})
\begin{equation}
\label{three_eqnsa}
\begin{array}{cccl}
-\Delta \vec{\sigma} + \nabla_s \vec{\mu}  = \lambda \vec{\sigma} + \sum_{p=1}^N \nu_p \vec{\phi}_p & \text{ and } & \mbox{div} \, \vec{\sigma} = \vec{0} & \text{ in} \hspace{1mm} \Omega,\\
\vec{\sigma n} = \vec{0} & \text{ and } & \nabla_n \vec{\sigma} \cdot (\vec{t} \otimes \vec{t}) = 0 & \text{ on} \hspace{1mm} \partial \Omega,
\end{array}
\end{equation}
along with Eq.\ \ref{ortnew}. Since the new extremization problem is posed on a nonempty subspace, it is reasonable to suppose that it has at least one solution $\tilde{\vec{\sigma}} \in {\mathcal S}_{N\perp}$ with associated $\tilde{\vec{\mu}}$, $\tilde{\lambda}$ and $\tilde{\nu_p}$, i.e.,
\begin{equation}
\label{nu1} -\Delta \tilde{\vec{\sigma}} + \nabla_s \tilde{\vec{\mu}}  = \tilde{\lambda} \tilde{\vec{\sigma}} + \sum_{p=1}^N \tilde{\nu}_p \vec{\phi}_p.
\end{equation}
The proof of existence of an extremizer in ${\mathcal S}_{N\perp}$ is technical and is presented in appendix \ref{snperp}.

Consider any eigenfunction $\vec{\phi}_k$, $1 \le k \le N$. Compute the inner product of
Eq.\ \ref{nu1} with $\vec{\phi}_k$. By the reasoning in appendix \ref{math_reductions}, the $\nabla_s \tilde{\vec{\mu}}$ term drops out. By Eq.\ \ref{ortnew}, the $\tilde{\lambda} \tilde{\vec{\sigma}}$ term drops out. By orthonormality of the eigenfunctions obtained so far, $\sum_{p=1}^N \tilde{\nu}_p \vec{\phi}_p$ contributes just $\tilde{\nu}_k$. By the manipulations that led to Eq.\ \ref{ip2}, the inner product thus becomes
\begin{equation}
\label{nu2}
\int_{\Omega}  \nabla \tilde{\vec{\sigma}} \cdot \nabla  \vec{\phi}_k  \, dA = \tilde{\nu}_k .
\end{equation}
However, since $\tilde{\vec{\sigma}}$ is an element of ${\cal S}$ and also orthogonal to $\vec{\phi}_k$, Eq.\ \ref{ip2} shows that
\begin{equation}
\label{nu3}
\int_{\Omega} \nabla  \vec{\phi}_k  \cdot \nabla \tilde{\vec{\sigma}}  \, dA = 0.
\end{equation}
Thus $\tilde{\nu}_k = 0$ for $1 \le k \le N$. Inserting these zeros in Eq.\ \ref{three_eqnsa}
we obtain exactly Eq.\ \ref{three_eqns}, which shows that the new solution merely adds another element to the existing sequence.
We conclude that there are infinitely many eigenfunctions.

It can now be shown that these eigenfunctions form a basis for
$\bar{\mathcal{S}}$, as follows.
Let ${\mathcal S}_{\infty}$ be the subspace spanned by the infinite sequence $\left \{ \vec{\phi}_p \right \}$, $p = 1, 2, \cdots$, with all eigenfunctions included.
If indeed there is an element of ${\mathcal{S}}$ that is not in ${\mathcal S}_{\infty}$, then arguments in the same spirit as above establish that this element merely adds
one more eigenfunction to the sequence, giving a contradiction (for details, see appendix \ref{infinity}). Finally, since every element of ${\mathcal{S}}$ can be expressed to arbitrary closeness in the $L^2$ norm as a linear combination
of our basis functions, so can every element of the closure $\bar{\mathcal{S}}$.
We conclude that our eigenfunctions provide a basis for residual stress states, as claimed\footnote{ One might, in some cases, consider self-equilibrating stresses under prescribed non-zero boundary tractions. In  such cases, the total stress $\vec{\sigma}$ can be written as the sum of a general traction free residual stress $\vec{\sigma}_h$ and {\em any} particular self-equilibrating $\vec{\sigma}_p$,  consistent with the applied tractions, and computed in any way we like. Our basis can then be used to represent $\vec{\sigma}_h$.}. Numerical examples presented below will provide ample empirical evidence of the same.

\section{Computation of eigenfunctions} 
In general, it is not possible to solve the eigenvalue problem of Eq.\ \ref{three_eqns} analytically. We have first computed some finite element solutions
for understanding, and then computed a large number of eigenfunctions for an annular domain using a semi-numerical approach. These are presented in the next two subsections.

\subsection{Finite element solutions}
To solve the eigenvalue problem using finite elements, we discretise the domain using eight-noded quadrilateral serendipity elements, such that the stress components are piecewise cubic \cite{zienkiewicz}. We note from Eq.\ \ref{three_eqns} that lower smoothness is required for $\mu_x$ and $\mu_y$, and we approximate them as piecewise constant. Details of the finite element procedure are given in appendix \ref{FEM}. 

Some eigenfunctions thus obtained are shown for three domains: an annular domain ($r_i = 0.1$, $r_o = 0.3$), a unit square, and a somewhat arbitrarily shaped, comparably sized planar domain: see figures \ref{annular_modes}, \ref{square_modes} and \ref{arbitrary_shapes} respectively. The mesh used was refined until the first several eigenvalue estimates were varying within tiny fractions of one percent.
\begin{figure}[t!]
	\centering
	\includegraphics [width=1\textwidth]{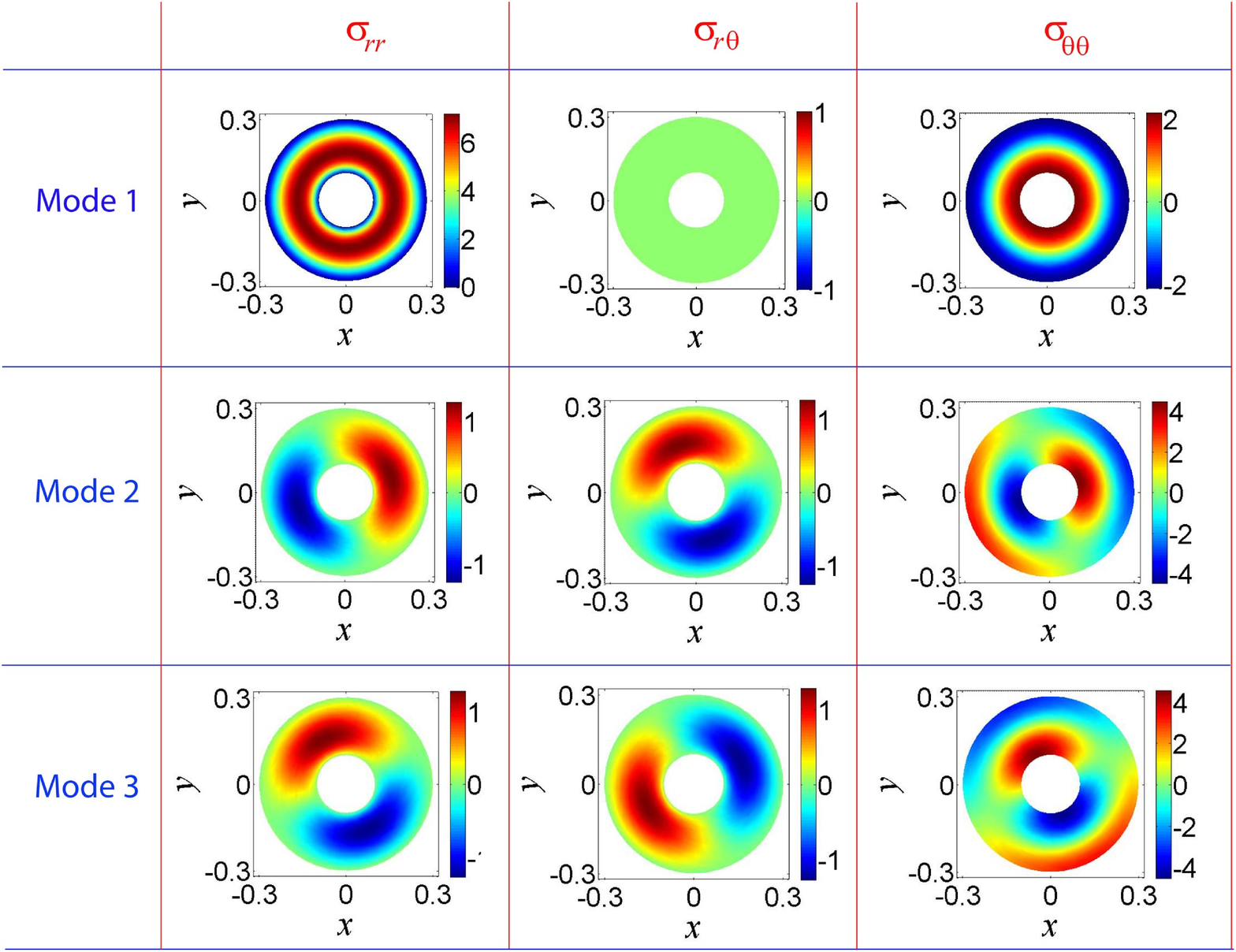}
	%	\vspace{-30mm}
	\caption{First three eigenfunctions for an annular domain; $\lambda_1=293.34$, $\lambda_2=\lambda_3=348.76$.}	
	\label{annular_modes}
	\vspace{5mm}
	\includegraphics [width=0.9\textwidth]{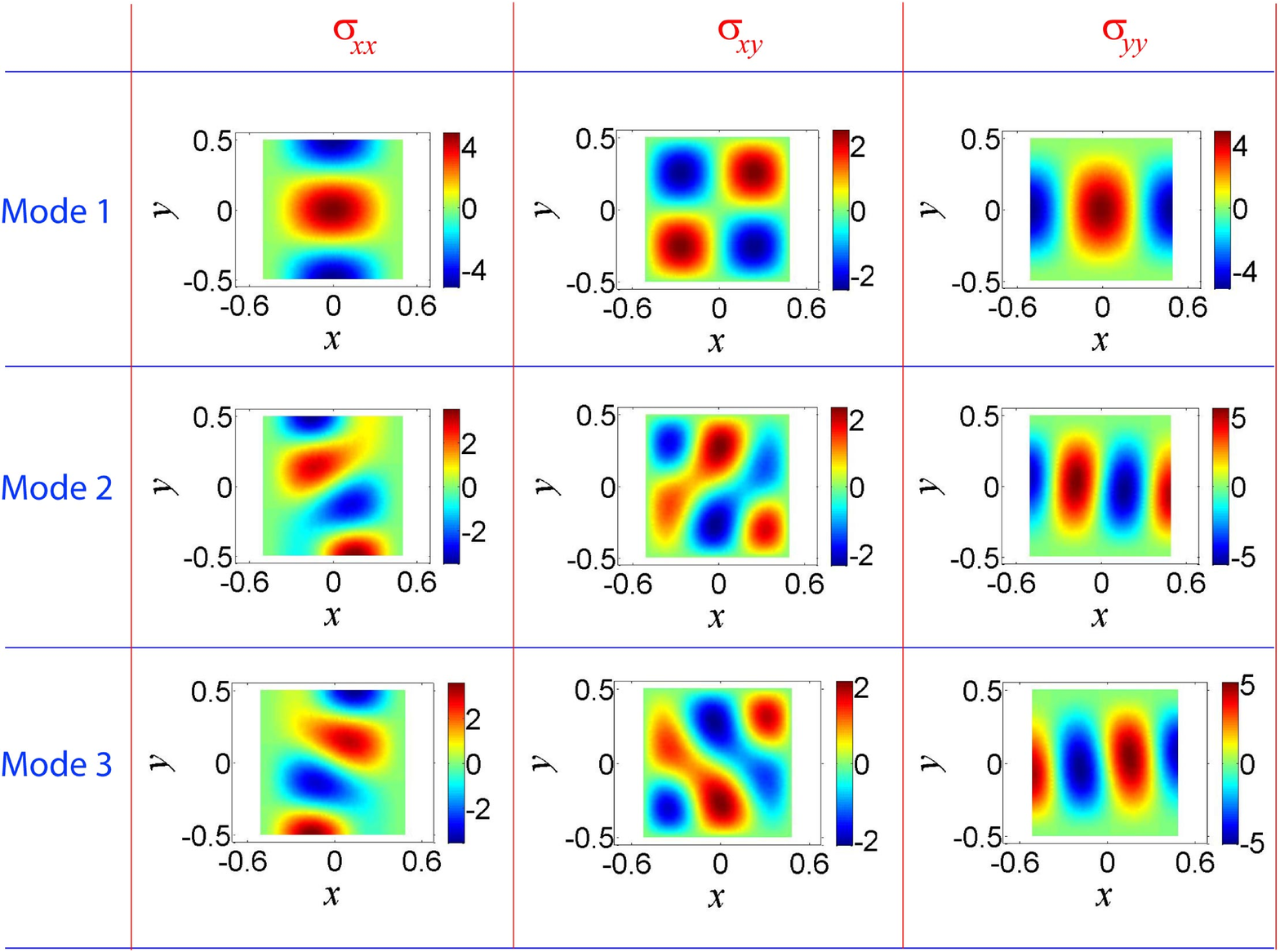}
	\caption{First three eigenfunctions for a square domain; $\lambda_1=59.12$, $\lambda_2 = \lambda_3=103.98$.}	
	\label{square_modes}
\end{figure}
\begin{figure}[t!]
	\centering
	\includegraphics [width=0.9\textwidth]{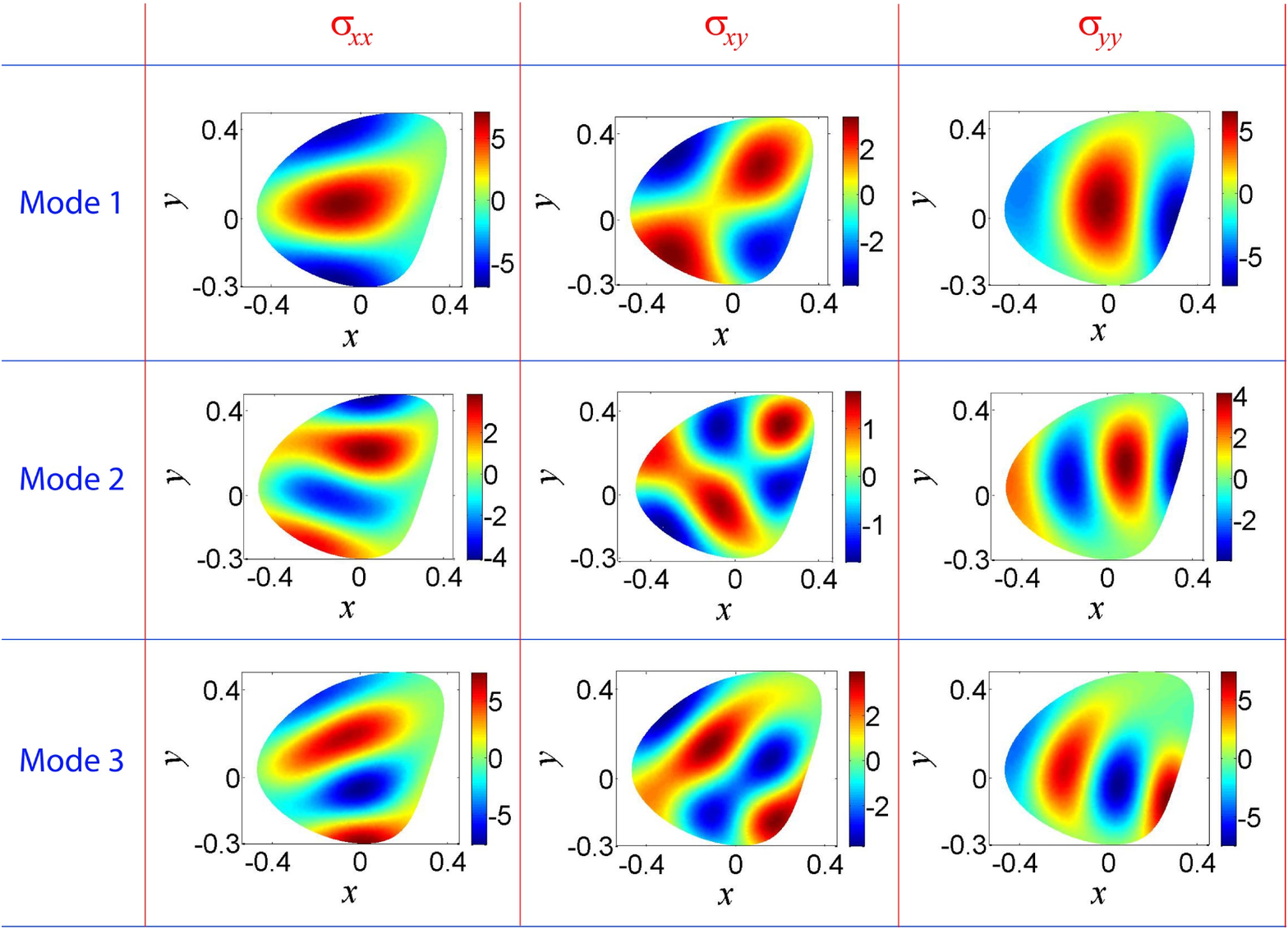}
	%	\vspace{-30mm}
	\caption{First three eigenfunctions for an arbitrarily shaped domain; $\lambda_1=99.50$, $\lambda_2=172.78$, $\lambda_3=200.27$.}	
	\label{arbitrary_shapes}
\end{figure}

\setlength{\textfloatsep}{5pt}
\subsection{Semi-analytical solutions for an annular domain}
On an annular domain, upon choosing a circumferential wavenumber $m$, the eigenvalue problem retains one independent variable ($r$). Many eigenfunctions can then be computed with great accuracy using a large number of $r$-points. For the numerical examples of stress interpolation presented in the next section, therefore, we use such eigenfunctions.

We consider an annular domain $\Omega$, centered at the origin, with inner radius $r_i = 0.1$ and outer radius $r_o = 0.3$. We denote the fields $\vec{\sigma}$ and $\vec{\mu}$ in polar coordinates as
$$ \vec{\sigma} = \sigma_{rr} (r,\theta) \, \vec{e_r}\otimes \vec{e_r} + \sigma_{r\theta} (r,\theta) \, \left(\vec{e_{r}}\otimes \vec{e_{\theta}} +  \vec{e_{\theta}}\otimes \vec{e_{r}}\right) + \sigma_{\theta\theta} (r,\theta) \, \vec{e_{\theta}}\otimes \vec{e_{\theta}}, $$
$$ \vec{\mu}=\mu_r (r,\theta) \vec{e_r} + \mu_{\theta} (r,\theta) \vec{e_{\theta}}. $$ 
The equation $-\Delta \vec{\sigma} + \nabla_s \vec{\mu} = \lambda \vec{\sigma}$
yields
$$\frac{\partial^2 \sigma_{rr}}{\partial r^2} + \frac{1}{r^2}\frac{\partial^2 \sigma_{rr}}{\partial \theta^2}+\frac{1}{r} \frac{\partial \sigma_{rr}}{\partial r} -\frac{4}{r^2} \frac{\partial \sigma_{r \theta}}{\partial \theta} -\frac{2 \sigma_{rr}}{r^2}  + \frac{2 \sigma_{\theta \theta}}{r^2}  - \frac{\partial \mu_r}{\partial r} = \lambda \sigma_{rr},$$
$$ \frac{\partial^2 \sigma_{\theta \theta}}{\partial r^2} + \frac{1}{r^2}\frac{\partial^2 \sigma_{\theta \theta}}{\partial \theta^2} + \frac{1}{r} \frac{\partial \sigma_{\theta \theta}}{\partial r} + \frac{4}{r^2} \frac{\partial \sigma_{r \theta}}{\partial \theta} + \frac{2 \sigma_{rr}}{r^2} - \frac{2 \sigma_{\theta \theta}}{r^2}-\frac{1}{r} \frac{\partial \mu_{\theta}}{\partial \theta}- \frac{\mu_r}{r} =\lambda \sigma_{\theta \theta},$$
and
$$ \frac{\partial^2 \sigma_{r \theta}}{\partial r^2} + \frac{1}{r^2}\frac{\partial^2 \sigma_{r \theta}}{\partial \theta^2} + \frac{1}{r} \frac{\partial \sigma_{r\theta}}{\partial r} + \frac{2}{r^2} \frac{\partial \sigma_{rr}}{\partial \theta} - \frac{2}{r^2} \frac{\partial \sigma_{\theta \theta}}{\partial \theta} -\frac{4 \sigma_{r \theta}}{r^2} + \frac{\mu_{\theta}}{2 r} - \frac{1}{2r} \frac{\partial \mu_r}{\partial \theta} -\frac{1}{2} \frac{\partial \mu_{\theta}}{\partial r} =\lambda \sigma_{r \theta}.$$
The equilibrium equation $ \mbox{div} \, \vec{\sigma} = \vec{0} $
becomes
$$ \frac{\partial \sigma_{rr}}{\partial r} + \frac{1}{r} \frac{\partial \sigma_{r \theta}}{\partial \theta} + \frac{\sigma_{rr} - \sigma_{\theta \theta}}{r}=0$$ 
and
$$ \frac{\partial \sigma_{r \theta}}{\partial r} + \frac{1}{r} \frac{\partial \sigma_{\theta \theta}}{\partial \theta} + \frac{2 \sigma_{r \theta}}{r}=0.$$
The boundary condition $ \vec{\sigma} \vec{n} = \vec{0}$ 
gives four scalar equations,
$$ \sigma_{rr}=0 \mbox{ at }  r=r_i \mbox{ and } r_o; \quad \sigma_{r\theta}=0 \mbox{ at }  r=r_i \mbox{ and } r_o.$$  
The natural boundary condition $ \nabla_n \vec{\sigma} \cdot (\vec{t} \otimes \vec{t}) = 0$ 
gives two scalar equations,
$$ \frac{\partial \sigma_{\theta \theta}}{\partial r}=0 \mbox{ at }  r=r_i \mbox{ and } r_o. $$
We now choose a wavenumber $m$ (any whole number).
Substituting
\begin{equation}
\label{eqsub}
\begin{split}
\sigma_{rr}=\widetilde{\sigma}_{rr}(r) \cos{m \theta}, \, \sigma_{\theta \theta}=\widetilde{\sigma}_{\theta \theta}(r) \cos{m \theta}, \, \sigma_{r\theta}=\widetilde{\sigma}_{r\theta}(r) \sin{m \theta},\\ \mu_r=\widetilde{\mu}_r (r) \cos{m \theta} \,\, \mbox{ and } \,\, \mu_{\theta}=\widetilde{\mu}_{\theta} (r) \sin{m \theta}
\end{split}
\end{equation}
in the above partial differential equations (PDEs), we obtain the following five ordinary differential equations (ODEs):
\begin{align}
\label{5eqns}
\begin{split}
\widetilde{\sigma}_{rr}''-\frac{m^2}{r^2} \widetilde{\sigma}_{rr} +\frac{\widetilde{\sigma}_{rr}'}{r}-\frac{4m\widetilde{\sigma}_{r \theta}}{r^2}-\frac{2\widetilde{\sigma}_{rr}}{r^2}+\frac{2\widetilde{\sigma}_{\theta \theta}}{r^2}-\widetilde{\mu}_r'+\lambda \widetilde{\sigma}_{rr}=0,\\
\widetilde{\sigma}_{\theta \theta}''-\frac{m^2}{r^2} \widetilde{\sigma}_{\theta \theta} +\frac{\widetilde{\sigma}_{\theta \theta}'}{r} + \frac{4m\widetilde{\sigma}_{r \theta}}{r^2}+\frac{2\widetilde{\sigma}_{rr}}{r^2}-\frac{2\widetilde{\sigma}_{\theta \theta}}{r^2}-\frac{m \widetilde{\mu}_{\theta}}{r}-\frac{\widetilde{\mu}_r}{r}+\lambda \widetilde{\sigma}_{\theta \theta}=0, \\
\widetilde{\sigma}_{r \theta}''-\frac{m^2}{r^2} \widetilde{\sigma}_{r \theta} +\frac{\widetilde{\sigma}_{r \theta}'}{r} -\frac{2m\widetilde{\sigma}_{rr}}{r^2}+\frac{2m\widetilde{\sigma}_{\theta \theta}}{r^2}-\frac{4 \widetilde{\sigma}_{r\theta}}{r^2}+\frac{\widetilde{\mu}_{\theta}}{2r}+\frac{m\widetilde{\mu}_r}{2r}-\frac{\widetilde{\mu}_{\theta}'}{2}+\lambda \widetilde{\sigma}_{r \theta}=0, \\
\widetilde{\sigma}_{rr}'+\frac{m\widetilde{\sigma}_{r\theta}}{r}+\frac{\widetilde{\sigma}_{rr}-\widetilde{\sigma}_{\theta \theta}}{r}=0, \\
\widetilde{\sigma}_{r \theta}'-\frac{m\widetilde{\sigma}_{\theta\theta}}{r}+\frac{2\widetilde{\sigma}_{r\theta}}{r}=0,
\end{split}
\end{align}
where primes denote $r$-derivatives, and we have suppressed the $r$-dependence of the field variables. Equations \ref{5eqns} have the structure of differential algebraic equations, and the last two were differentiated once each for setting up as a system of first order ODEs.
Introducing the new variable $\vartheta$, we obtain the following six first order ODEs (with $\widetilde{\sigma}_{rr}''$ and $\widetilde{\sigma}_{r \theta}''$ eliminated):
\begin{equation}
\label{6eqns}
\begin{split}
\widetilde{\sigma}_{rr}'=-\frac{\widetilde{\sigma}_{rr}}{r}-\frac{m \widetilde{\sigma}_{r\theta}}{r} + \frac{\widetilde{\sigma}_{\theta\theta}}{r},\\
\widetilde{\sigma}_{r\theta}'=-\frac{2\widetilde{\sigma}_{r\theta}}{r}+\frac{m\widetilde{\sigma}_{\theta\theta}}{r},\\
\widetilde{\sigma}_{\theta\theta}'=\vartheta,\\
\vartheta'=\frac{m^2 \widetilde{\sigma}_{\theta \theta}}{r^2}-\frac{\vartheta}{r}-\frac{4m\widetilde{\sigma}_{r\theta}}{r^2}-\frac{2\widetilde{\sigma}_{rr}}{r^2}+\frac{2\widetilde{\sigma}_{\theta\theta}}{r^2}+\frac{\widetilde{\mu}_r}{r}+\frac{m \widetilde{\mu}_{\theta}}{r}-\lambda \widetilde{\sigma}_{\theta\theta},\\	\widetilde{\mu}_r'=-\frac{(m^2-1)\widetilde{\sigma}_{\theta\theta}}{r^2}-\frac{m\widetilde{\sigma}_{r\theta}}{r^2}+\frac{\vartheta}{r}-\frac{(m^2+1)\widetilde{\sigma}_{rr}}{r^2}+\lambda \widetilde{\sigma}_{rr},\\
\widetilde{\mu}_{\theta}'=\frac{2m\vartheta}{r}-\frac{2m^2\widetilde{\sigma}_{r\theta}}{r^2}-\frac{4m\widetilde{\sigma}_{rr}}{r^2}+\frac{m\widetilde{\mu}_{r}}{r}+\frac{\widetilde{\mu}_{\theta}}{r}+2 \lambda \widetilde{\sigma}_{r\theta}.
\end{split}
\end{equation}
We already have six homogeneous boundary conditions, three at $r_i$ and three at $r_o$.
Nonzero solutions will be possible only for specific discrete values of $\lambda$, which must also be determined as part of the solution; but the eigenfunctions will be arbitrarily scalable.
To make things definite, we introduce a normalizing boundary condition,
$$ \widetilde{\sigma}_{\theta \theta}=1\mbox{ at } r=r_i.$$
We have solved the above eigenvalue problem repeatedly using Matlab's built-in routine `bvp4c' as well as alternative numerical routines of our own (based on the Newton-Raphson method with numerically estimated Jacobians), for our chosen $m$. Each solution obtained gives one eigenvalue-eigenfunction pair. Initial values must be chosen to ensure that all eigenfunctions are obtained and none missed. The foregoing finite element solutions help identify the first one or two for any $m$;
for the higher modes, plots of $\lambda_p$ against $p$ help to identify missed eigenvalues, as does
counting the number of zero crossings of $\widetilde{\sigma}_{\theta \theta}$.

For demonstration, we choose $m=3$. The radial variation of stress component functions $\widetilde{\sigma}_{rr}$, $\widetilde{\sigma}_{r \theta}$ and $\widetilde{\sigma}_{\theta \theta}$ for the first three eigenfunctions are shown in figure \ref{m3}. 
\begin{figure}[t!]
	\centering
	\includegraphics [scale=0.4]{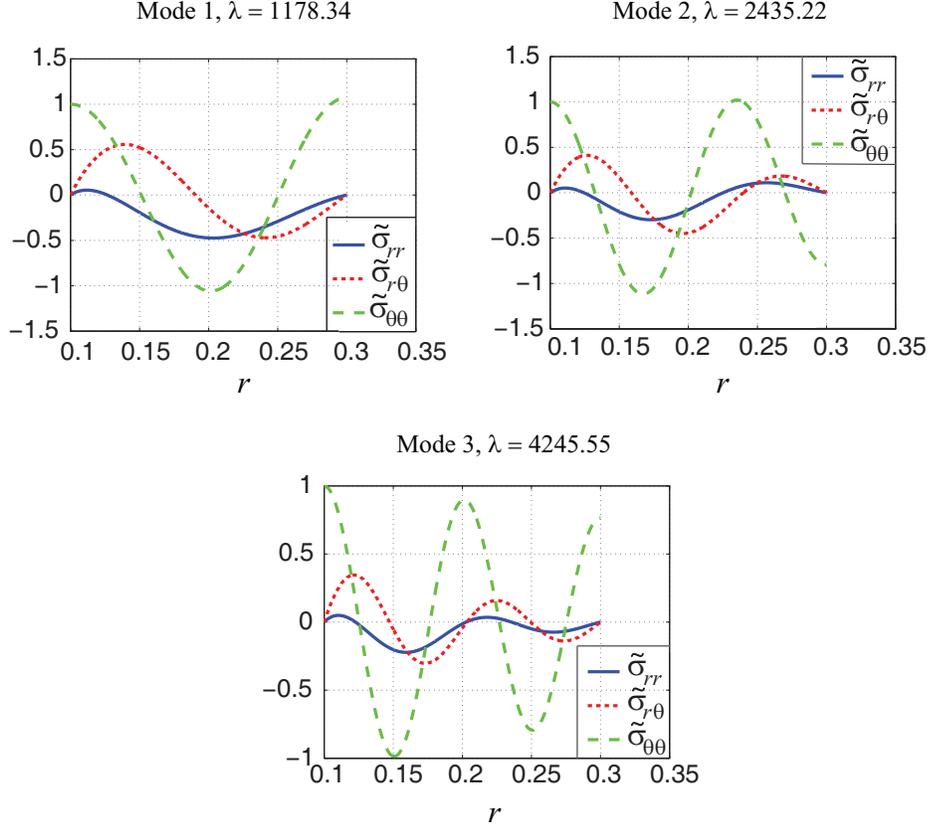}
	\caption{First three eigenfunctions for the annular domain, with $m=3$. In the finite element solution, these are mode numbers (9,10), (26,27) and (55,56). The eigenvalues from the semi-analytical approach and FEM match near-perfectly.}
	\label{m3}
\end{figure}

Finally, for $m>0$, all eigenvalues appear in pairs; and for each eigenfunction obtained above, we can obtain another one by taking the partial derivative with respect
to $\theta$ in Eq.\ \ref{eqsub} and then dividing by $m$.

With this semi-analytical approach on the annular domain, for given $m$, we can accurately compute, say, 50 eigenfunctions. Obtaining the same number
of $m=3$ eigenfunctions from the finite element approach would require computation of thousands of eigenfunctions with many different wave numbers.

In the above calculation, we have not normalized the eigenfunctions to unit norm, but that has no real consequence below.
We now turn to demonstrations
of fitting several self-equilibrating traction-free stress fields on the annular region. A numerical example based on a metal forming simulation
is presented in appendix \ref{rolling}.

\section{Examples of fitting residual stress fields} \label{An}
In this section, we consider a few candidate residual fields on an annular domain and fit them using the eigenfunctions computed above. For simplicity, we consider residual stresses $\vec{\sigma}$ involving a single circumferential wave number $m$, with components given by
\begin{equation}
\label{form}
\vec{\sigma}=\sigma_{rr} \cos{m \theta} \, \, \vec{e_{r}} \otimes \vec{e_r} + \sigma_{r\theta} \sin{m \theta} \, \, \left(\vec{e_{r}} \otimes \vec{e_{\theta}} +  \vec{e_{\theta}} \otimes \vec{e_r} \right) + \sigma_{\theta \theta} \cos{m \theta} \, \, \vec{e_{\theta}} \otimes \vec{e_{\theta}},
\end{equation}
where the $r$-dependence of the stress components has been suppressed (note the similarity with Eq.\ \ref{eqsub}). We begin with
\begin{equation}
\label{resi}
\vec{\sigma} = \sum_{i=1}^{\infty} a_i \vec{\phi}_i,
\end{equation} 
where the eigenfunctions $\vec{\phi}_i$ were obtained above using the semi-analytical approach. Using the orthogonality of $\vec{\phi}_i$, we have
\begin{equation}
\nonumber
a_i=\frac{\int_{\Omega} \vec{\sigma} \cdot \vec{\phi}_i \, dA}{\int_{\Omega} \vec{\phi}_i \cdot \vec{\phi}_i \, dA},
\end{equation}
where the denominator would be unity if we had normalized our eigenfunctions.
Truncating the series in Eq.\ \ref{resi}, we write
\begin{equation}
\label{trunc}
\vec{\sigma}_N = \sum_{i=1}^{N} a_i \vec{\phi}_i,
\end{equation}
and use the squared relative error measure
\begin{equation}
\nonumber
E_N=\frac{\int_{\Omega} \left(\vec{\sigma}-\vec{\sigma}_N\right) \cdot \left(\vec{\sigma}-\vec{\sigma}_N\right) \, dA}{\int_{\Omega} \vec{\sigma} \cdot \vec{\sigma} \, dA}
\end{equation}
to study convergence in the norm of Eq. \ref{norm}. 

We now present four examples of candidate residual stress fields, and the corresponding fits. In the first two examples we construct hypothetical residual stress fields directly, with wavenumber $m=3$, from the equilibrium equations. In the third example we use the stress field in two concentric elastic cylinders in a shrink fit, with $m=0$. In the fourth example we consider the thermoelastic stress state in an initially-unstressed elastic annular body subjected to a subsequent
nonuniform rise in temperature, with $m=3$.

\subsection{Example 1: hypothetical stress field, \boldmath{$m=3$}}
\label{hypsf}
Let $\vec{\sigma}$ be as given in Eq.\ \ref{form}, with $m=3$. From equilibrium,
\begin{equation}
\label{eq}
\begin{split}
\sigma_{rr}'+\frac{m\sigma_{r\theta}}{r}+\frac{\sigma_{rr}-\sigma_{\theta \theta}}{r}=0, \\
\sigma_{r\theta}'-\frac{m\sigma_{\theta\theta}}{r}+\frac{2\sigma_{r\theta}}{r}=0,
\end{split}
\end{equation}
with four boundary conditions:
$$\sigma_{rr} = \sigma_{r\theta} = 0, \mbox{ at } r = r_i \mbox{ and } r = r_o.$$

\begin{figure}[t!]
	\centering
	\includegraphics [scale=0.45]{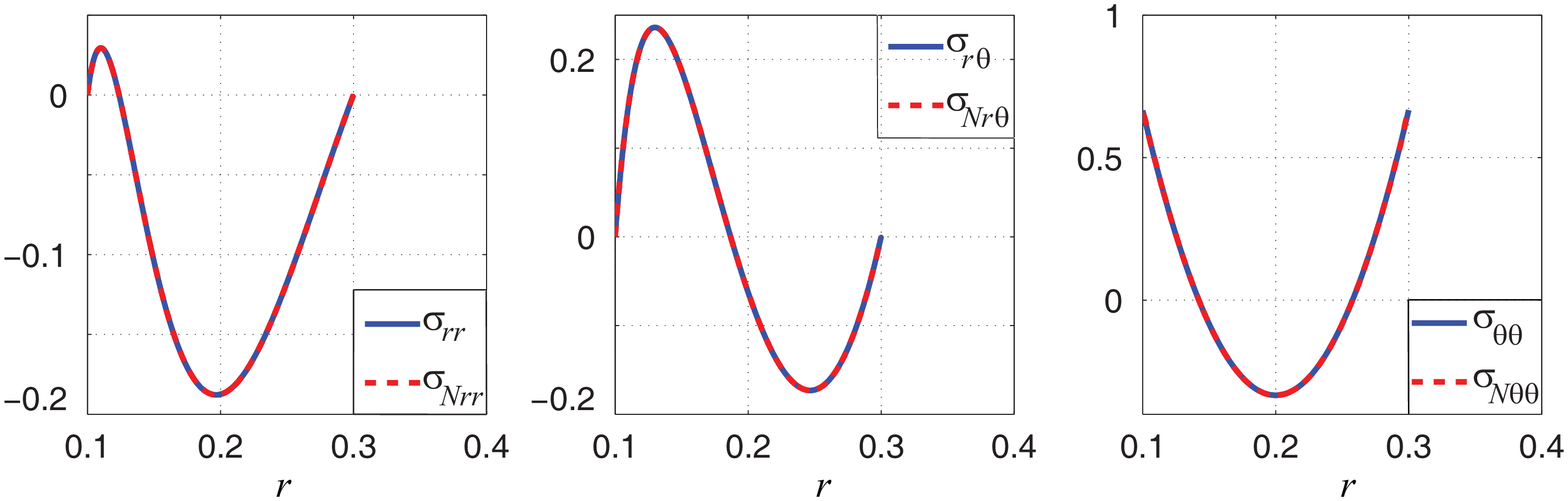}
	\caption{True and fitted stress fields $\vec{\sigma}$ and $\vec{\sigma}_N$ of Example 1, with $N=50$. }
	\label{sigmaN_hyp1}
	\vspace{5mm}
	\includegraphics [scale=0.3]{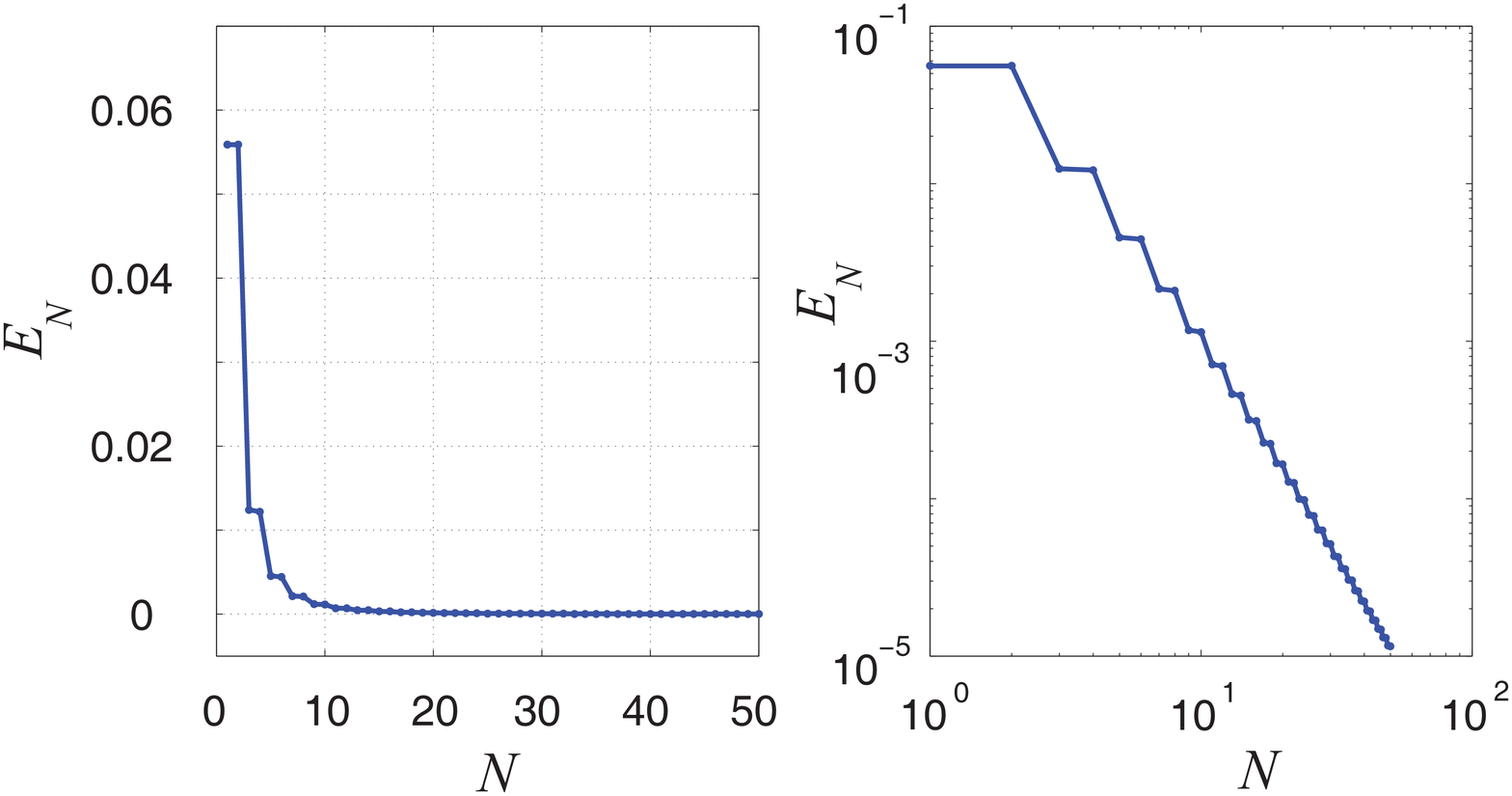}
	\caption{$E_N$ versus $N$, Example 1. Left: linear scale; right: log-log scale. Compare with figure \ref{En_fvm}.}
	\label{En_hyp1}
\end{figure}

To construct hypothetical residual stress fields, we can assume an arbitrary functional form 
$$ \sigma_{\theta\theta}=A(r)$$ 
with {\em two} free parameters in it. We can then solve for $\sigma_{r\theta}$ from the second of Eqs.\ \ref{eq}, retaining an integration constant. We finally solve for $\sigma_{rr}$ from the first of Eqs.\ \ref{eq}, retaining one more integration constant. The four boundary conditions can be satisfied using the two integration constants along with the two free parameters in $A(r)$.
We show two specific examples of stress fields computed using this approach.

\begin{figure}[t!]
	\centering
	\includegraphics [scale=0.45]{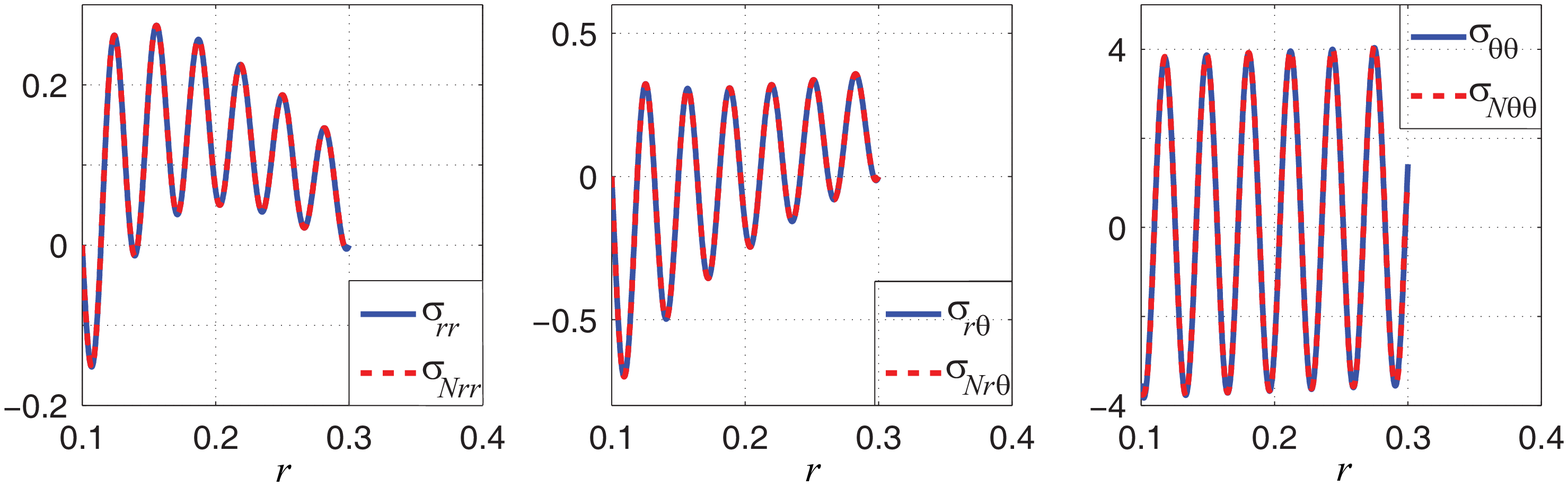}
	\caption{True and fitted stress fields $\vec{\sigma}$ and $\vec{\sigma}_N$ of Example 2, with $N=50$.}
	\label{sigmaN_hyp2}
%	\vspace{-5mm}
	\includegraphics [scale=0.3]{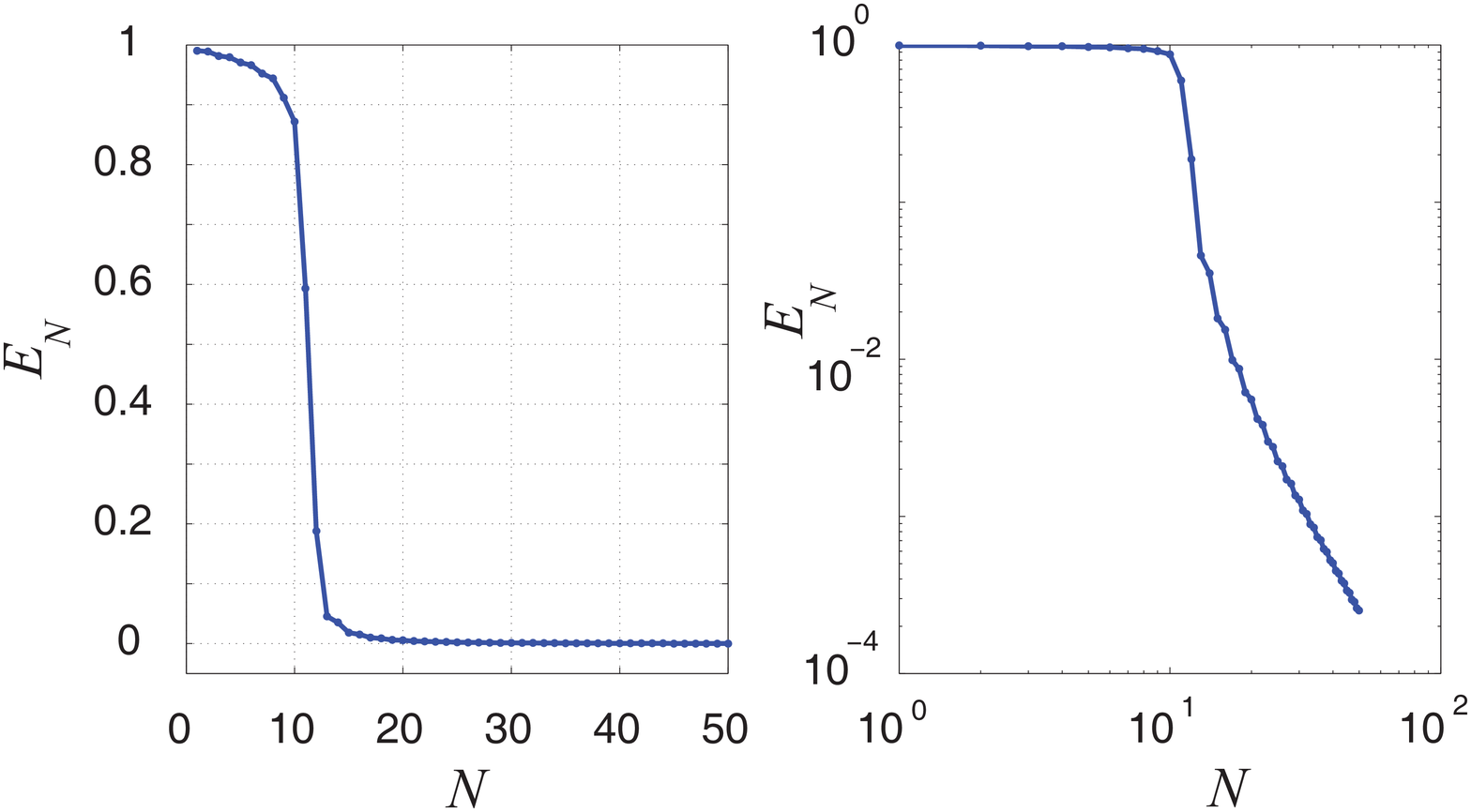}
	\caption{$E_N$ versus $N$, Example 2. Left: linear scale; right: log-log scale.}
	\label{En_hyp2}
\end{figure}

For the first example, we choose
$$ A(r)=C_0 + C_1 r + 100 r^2, $$
where $C_0$ and $C_1$ are free parameters, and the coefficient of 100 is arbitrary. Following the procedure above, we obtain $C_0=3.667$ and $C_1=-40.$ The resulting expressions for $\vec{\sigma}$ are given in appendix \ref{expressions}.
Figure \ref{sigmaN_hyp1} shows the components of $\vec{\sigma}$, along with components of the fitted $\vec{\sigma}_N$ ($N=50$). The error measure $E_N$ versus $N$ is plotted in figure \ref{En_hyp1}. Convergence is rapid, like $N^{-3}$ for large $N$, with $E_5 < 0.005$.

We mention that the $m=3$ normal vibration modes for the same domain (isotropic linear elasticity, plane strain) were computed separately and the stresses induced by those modes were also used in an attempted approximation of this same hypothetical stress field. The unsuccessful results of that attempt were plotted in figure \ref{En_fvm} (recall Eqs.\ \ref{demo}, further details omitted).

\subsection{Example 2: hypothetical stress field, \boldmath{$m=3$}}
For another example following section \ref{hypsf} above, we choose
$$A(r)=C_0 \sin(200r) + \frac{C_1}{r} + r.$$
The coefficient of 200 within the sine is chosen to produce several oscillations between $r_i = 0.1$ and $r_o = 0.3$. Calculations yield $C_0=-3.805$ and $C_1=-1.284 \times 10^{-2}$. The resulting expressions for $\vec{\sigma}$ are given in appendix \ref{expressions}. 
The fit (for $N$=50) is shown in figure \ref{sigmaN_hyp2}, and $E_N$ is plotted in figure \ref{En_hyp2}.

In figure \ref{En_hyp2} (left), we see that $E_{13}$ drops low. This is because, by choice, $\sigma_{\theta \theta}$ has 13 zero crossings. By figure \ref{m3}, we expect the $n^{\rm th}$ eigenfunction to have $n+1$ zero crossings in $\sigma_{\theta \theta}$. Therefore the 12$^{\rm th}$ eigenfunction has 13 zero crossings, and
$E_{13}$ is small. Subsequent convergence is rapid, like $N^{-3}$ for large $N$, with $E_{17} < 0.01$.

\subsection{Example 3: shrink fitted cylinder, \boldmath{$m=0$}} \label{shrink}
We consider an inner cylinder with inner radius $r_i$ and notional outer radius $r_c$, an outer cylinder with notional inner radius $r_c$ and outer radius $r_o=0.3$, with a small radial interference equal to $\delta$. The Young's modulus and Poisson's ratio of both cylinders are denoted by $E$ and $\nu$ respectively. The expressions for the resulting axisymmetric stress fields are given in appendix \ref{expressions}. We use eigenfunctions with $m=0$ in Eq.\ \ref{trunc}. 

Figure \ref{sigmaN_shrink} shows the nonzero components of $\vec{\sigma}$ and $\vec{\sigma}_N$ ($N$=100). Because $\sigma_{\theta \theta}$ is discontinuous at the contact surface between cylinders, convergence is slower (there are Gibbs oscillations \cite{gibbs1899fourier}). 
The plot of $E_N$ against $N$ in figure \ref{En_shrink} shows convergence like $N^{-1}$ for large $N$, with $E_{43} < 0.01$. Recalling the set $\mathcal{{S}}$ (Eq.\ \ref{eqdefS}) and its closure $\mathcal{\bar{S}}$, we note that $\vec{\sigma}$ belongs to $\mathcal{\bar{S}}$ but not $\mathcal{{S}}$. Convergence is still obtained because the $\vec{\phi}_i$ form a basis for $\mathcal{\bar{S}}$.
\begin{figure}[t!]
	\centering	
	\includegraphics [scale=0.45]{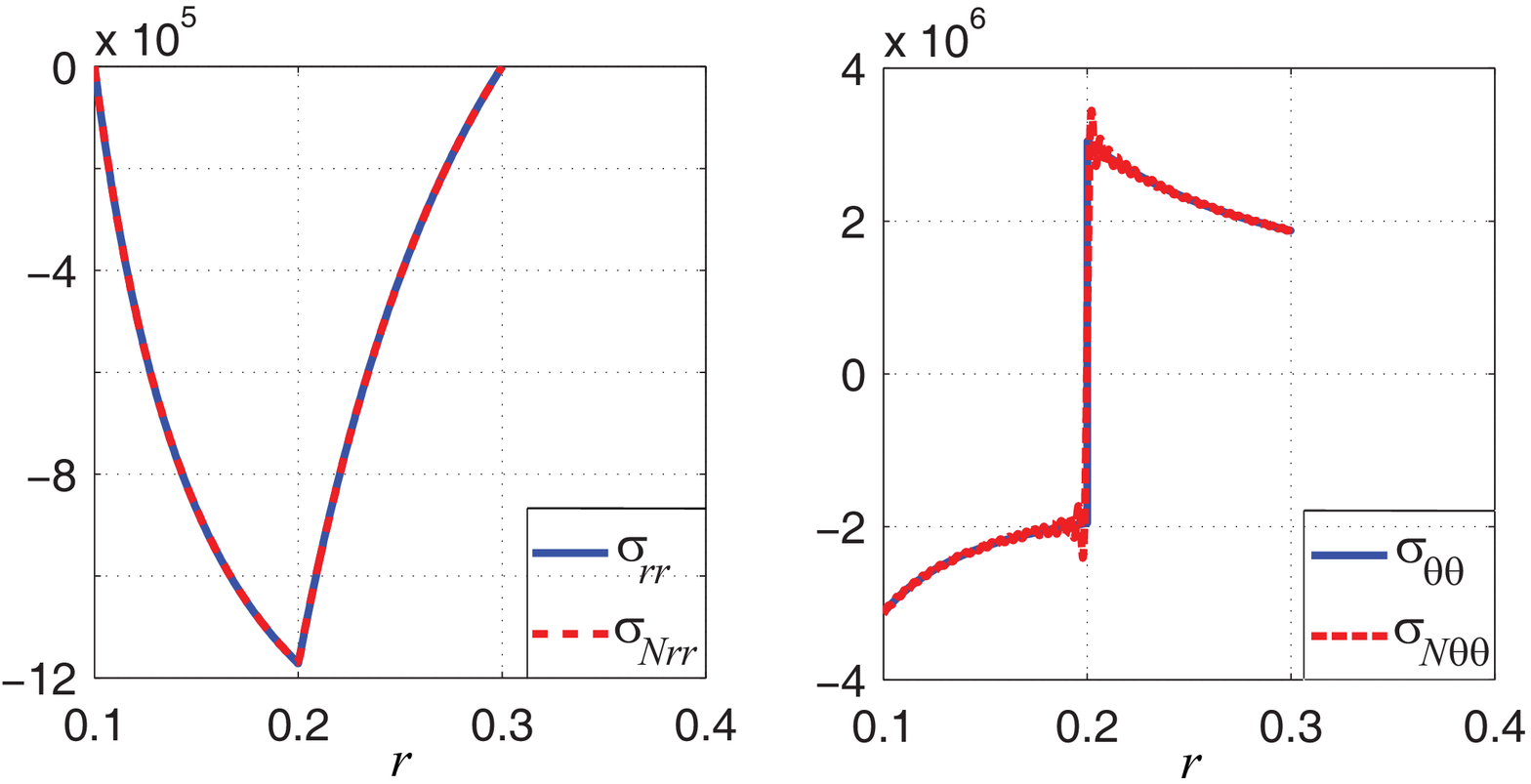}	
	\caption{True and fitted stress fields $\vec{\sigma}$ and $\vec{\sigma}_N$ of Example 3, with $N=100$.} 
	\label{sigmaN_shrink}
	\vspace{5mm}
	\centering
	\includegraphics [scale=0.3]{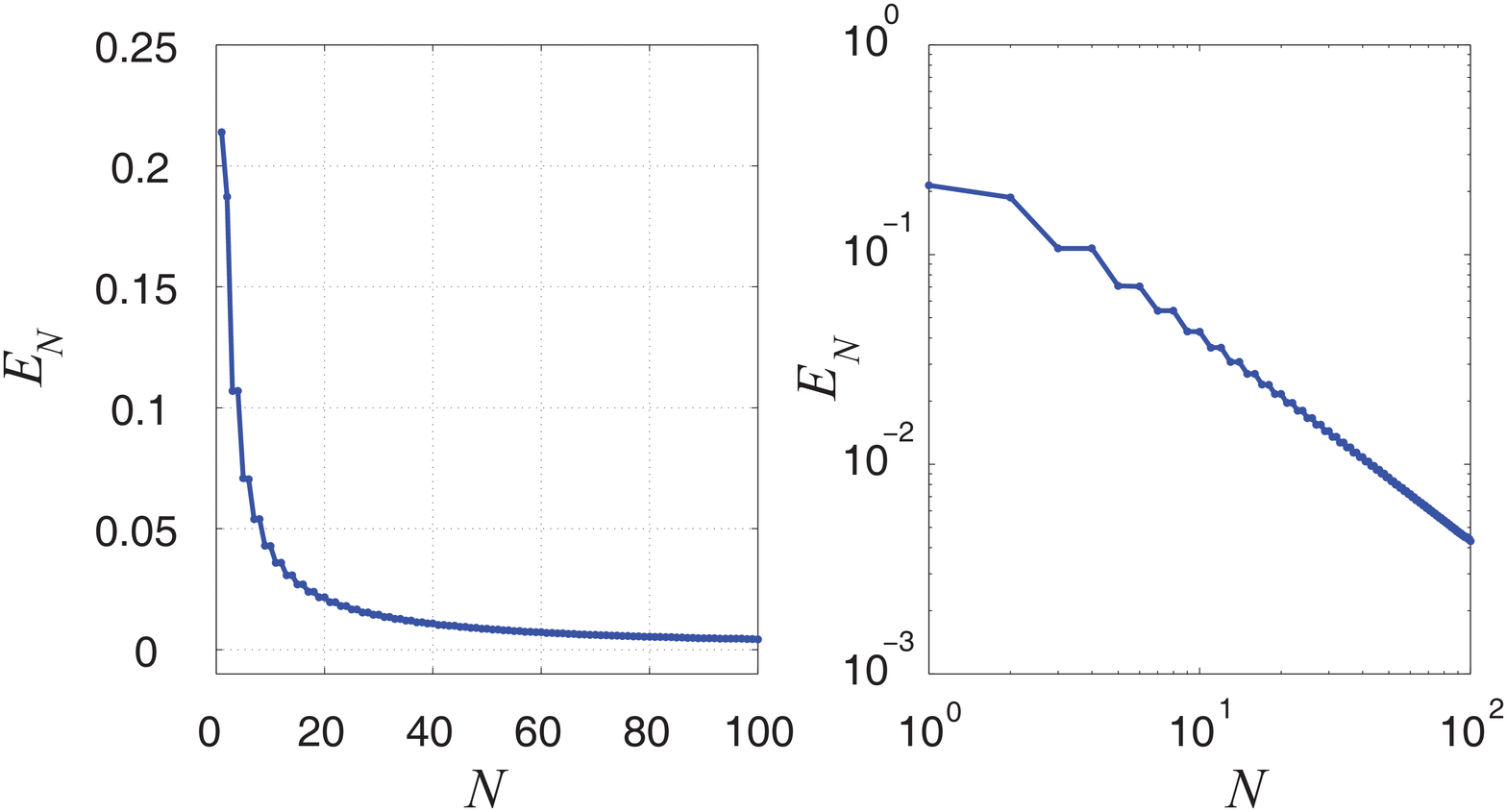}
	\caption{$E_N$ versus $N$, Example 3. Left: linear scale; right: log-log scale.}
	\label{En_shrink}
\end{figure}
\setlength{\textfloatsep}{5pt}

\subsection{Example 4: thermoelastic residual stress, \boldmath{$m=3$}}
If the initially unstressed annular unstressed region, with thermal coefficient $\alpha$, is subjected to a temperature change $T(r,\theta) = r \cos (3 \theta)$, the resulting thermal strain 
$$\vec{\varepsilon_{T}} = \alpha T \vec{I}$$
violates local compatibility, i.e., $\text{curl} \, \text{curl} \, \vec{\varepsilon_T} \neq \vec{0}$ (see e.g., \cite{barber}). The `global compatibility' equation derived from C\'esaro's integral \cite{boley}, for $m=3$, is trivially satisfied.
\begin{figure}[t!]
	\centering
	\includegraphics [scale=0.45]{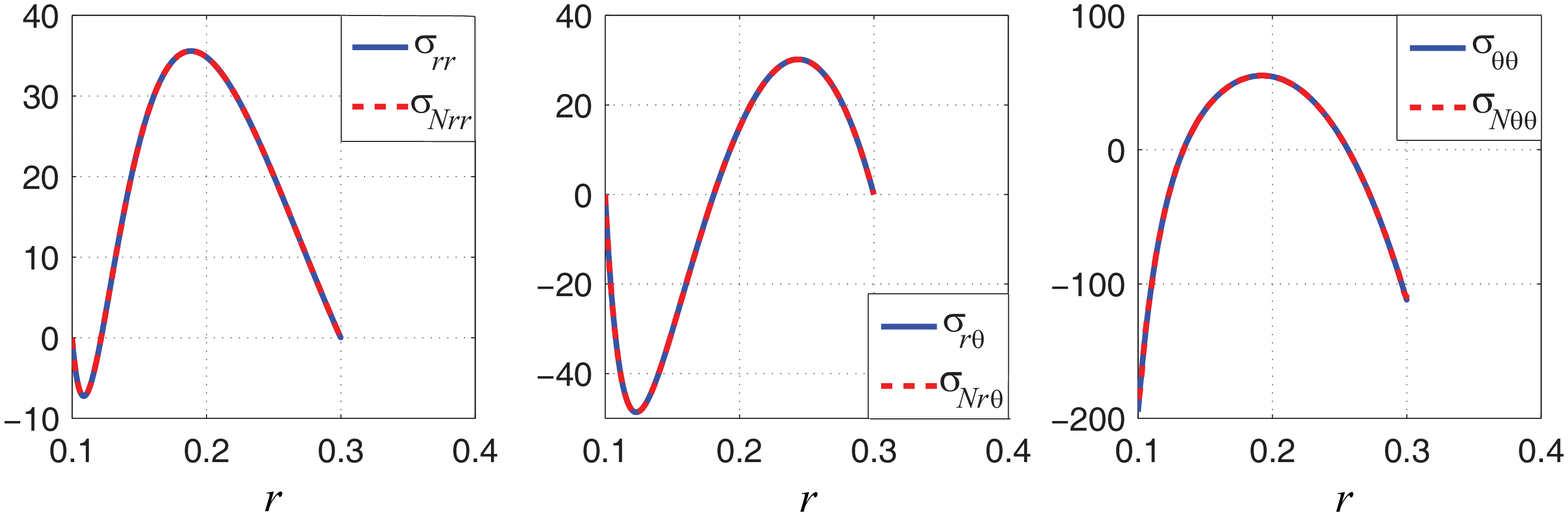}
	\caption{True and fitted thermoelastic stress fields $\vec{\sigma}$ (Eqs.\ \ref{thermode}, \ref{ODE2} and \ref{bc}) and $\vec{\sigma}_N$ of Example 4, with $N=50$.}
	\label{sigmaN_therm}
	\vspace{5mm}
	\includegraphics [scale=0.3]{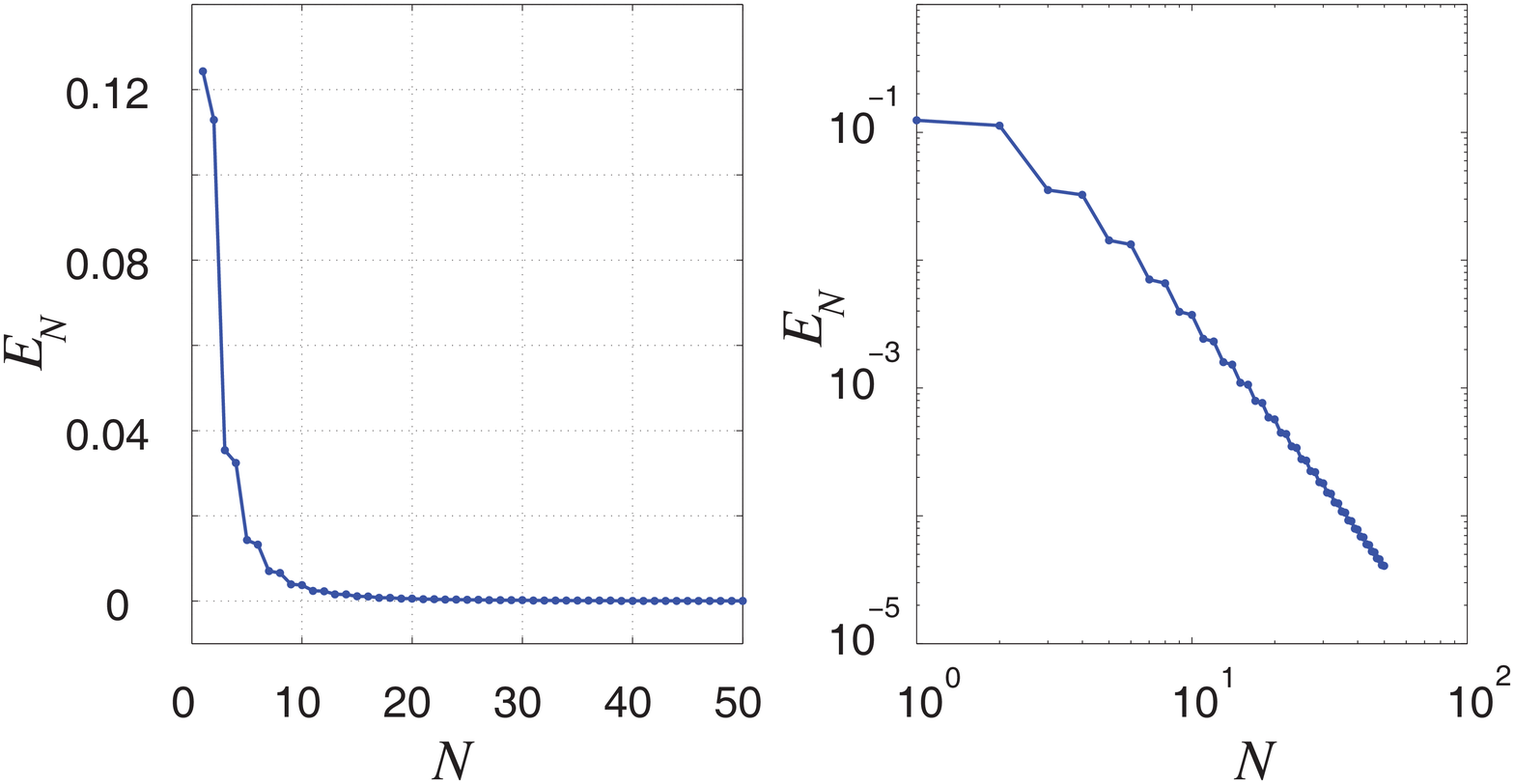}
	\caption{$E_N$ versus $N$, Example 4. Left: linear scale; right: log-log scale.}
	\label{En_therm}
\end{figure} 
The resulting stress $\vec{\sigma}$ satisfies (see e.g., \cite{boley})
\begin{equation}
\label{therm}
\Delta ( \text{Tr} \, \vec{\sigma}) = \frac{- \alpha E}{1-\nu} \Delta T,
\end{equation}
where $E$ is Young's modulus, $\nu$ is Poisson's ratio, and `Tr' denotes `trace.'
Substituting the expressions
\begin{equation}
\nonumber
\vec{\sigma}=\sigma_{rr} \cos{m \theta} \, \, \vec{e_{r}} \otimes \vec{e_r} + \sigma_{r\theta} \sin{m \theta} \, \, \left(\vec{e_{r}} \otimes \vec{e_{\theta}} +  \vec{e_{\theta}} \otimes \vec{e_r} \right) + \sigma_{\theta \theta} \cos{m \theta} \, \, \vec{e_{\theta}} \otimes \vec{e_{\theta}}
\end{equation}
in Eq.\ \ref{therm} gives
\begin{equation}
\label{thermode}
(\sigma_{rr} + \sigma_{\theta \theta})''+ \frac{(\sigma_{rr}+\sigma_{\theta \theta})'}{r} -\frac{m^2(\sigma_{rr}+\sigma_{\theta \theta})}{r^2}=- \frac{(m^2-1) \beta }{r},
\end{equation} 
where $\displaystyle \beta = \frac{-\alpha E}{(1-\nu)},$ and primes denote $r$-derivatives.
Eliminating $\sigma_{r \theta}$ from the equilibrium equations, we obtain another ODE:
\begin{equation}
\label{ODE2}
\sigma_{rr}''+ \frac{4\sigma_{rr}'}{r}-\frac{\sigma_{\theta \theta}'}{r}+\frac{2\sigma_{rr}}{r^2}+\frac{(m^2-2)\sigma_{\theta\theta}}{r^2}=0. \end{equation}
Traction free boundary conditions on the inner and outer radius, in terms of $\sigma_{rr}$ and $\sigma_{\theta \theta}$, are
\begin{equation}
\label{bc}
\begin{split}
\sigma_{rr}=0 \,\,\, \mbox{at} \,\,\, r=r_i \,\,\,\mbox{and} \,\, \, r_o, \\
\sigma_{rr}' +\frac{\sigma_{rr}-\sigma_{\theta\theta}}{r}=0 \,\,\, \mbox{at} \,\,\, r=r_i \,\,\,\mbox{and} \,\, \, r_o.
\end{split}
\end{equation}
The boundary value problem described by Eqs.\ \ref{thermode}, \ref{ODE2} and \ref{bc} can be solved numerically (iteratively; details omitted).

Figure \ref{sigmaN_therm} shows the components of $\vec{\sigma}$ and the fitted $\vec{\sigma}_N$ ($N$=50). Figure \ref{En_therm} shows $E_N$ versus $N$. Convergence is rapid as expected, with $E_7 < 0.01$.  

This concludes our demonstration of fitting reasonable but arbitrary, known, self-equilibrating, and traction free stress states (``residual stresses'') on an annular domain using the basis functions developed in this paper. For a different example of fitting a residual stress obtained from a metal forming simulation in Abaqus, please see appendix \ref{rolling}.

\section{Extension of the theory to three dimensions}
Our derivation of the eigenvalue problem in section \ref{PSsec} was for a two-dimensional domain. The extension of the theory
to three dimensions is straightforward, and is now presented for completeness. Computations, which will require finite element formulations in 3D, are left for future work.

Most of the development of section \ref{PSsec} is directly applicable to three dimensions if we interpret the ``$dA$'' in the domain integrals to be volume elements.
While obtaining Eq.\ \ref{three_eqns}, the two-dimensionality of the domain $\Omega$ was used only to derive the point-wise natural boundary condition of Eq.\ \ref{q2q3} from the integral condition of Eq.\ \ref{natural}. As a result, in three dimensions, only the fourth of Eqs.\ \ref{three_eqns} changes. 

Equation \ref{natural} in three dimensions is
\begin{equation}
\label{n3d}
\int_{\partial \Omega} (\nabla \vec{\sigma} \circ \vec{\zeta}) \cdot \vec{n} \, dS = 0,
\end{equation}
where ``$dS$'' is now interpreted as an infinitesimal area element on the surface $\partial \Omega$ of the three-dimensional domain $\Omega$.

Consider an arbitrarily small portion $\Delta S$ including any point $P$ on $\partial \Omega$. Restricting attention to
$\vec{\zeta}$ that is nonzero only on $\Delta S$, Eq.\ \ref{n3d} becomes
\begin{equation}
\label{n3d1}
\int_{\Delta S} (\nabla \vec{\sigma} \circ \vec{\zeta}) \cdot \vec{n} \, dS = 0.
\end{equation}
Equation\ \ref{n3d1} can be rewritten, using indicial notation as
\begin{equation}
\label{n3d2}
\int_{\Delta S} (\nabla \vec{\sigma} \circ \vec{\zeta}) \cdot \vec{n} \, dS = \int_{\Delta S} \sigma_{ij,k} \zeta_{ij} n_k \, dS = \int_{\Delta S} \sigma_{ij,k} n_k \zeta_{ij}  \, dS = \int_{\Delta S} \nabla_n  \vec{\sigma} \cdot \vec{\zeta} \, dS = 0.
\end{equation}
Since $\Delta S$ is arbitrarily small, and $\nabla_n \vec{\sigma}$, $\vec{\zeta}$ are continuous, we can use localization to conclude that
$$ \nabla_n \vec{\sigma} \cdot \vec{\zeta}=0 \, \, \mbox{on} \,\, \partial \Omega.$$
We choose a pair of convenient orthonormal vectors $\vec{t_1}$ and $\vec{t_2}$ in the tangent plane passing through $P$. This can be done, e.g., using the Cartesian
unit vector $\vec{e_1}$ as
$$ \vec{t_1}=\frac{\vec{e_1} \times \vec{n}}{\norm{\vec{e_1} \times \vec{n}}} \, \, \, \mbox{and} \, \, \, \vec{t_2}=\vec{n} \times \vec{t_1},$$
where `$\times$' represents the vector cross product; ($\vec{t_1}$,$\vec{t_2}$,$\vec{n}$) form a right handed orthonormal triad. If $\vec{n}$ is parallel, or almost parallel to $\vec{e_1}$, then $\vec{e_1}$ can be replaced by $\vec{e_2}$ in the subsequent discussion. 

Since $\vec{\zeta}$ is symmetric and satisfies $\vec{\zeta} \vec{n}=\vec{0}$, it must be expressible as
$$ \vec{\zeta} = \kappa_1 \, \vec{t_1} \otimes \vec{t_1} + \kappa_2 \, \vec{t_2} \otimes \vec{t_2} + \kappa_3 \left(\vec{t_1} \otimes \vec{t_2} + \vec{t_2} \otimes \vec{t_1}\right)$$
for arbitrary $\kappa_1, \kappa_2, \kappa_3$. 
First choosing $\kappa_2=\kappa_3=0$ and $\kappa_1\neq 0$, we obtain the natural boundary condition (compare with Eq.\ \ref{q3q4})
$$ \nabla_n \vec{\sigma} \cdot \left(\vec{t_1} \otimes \vec{t_1}\right)=0.$$
Similarly, we obtain two more natural boundary conditions:
$$ \nabla_n \vec{\sigma} \cdot \left(\vec{t_2} \otimes \vec{t_2}\right)=0 \,\,\, \mbox{ and } \,\,\, \nabla_n \vec{\sigma} \cdot \left(\vec{t_1} \otimes \vec{t_2} + \vec{t_2} \otimes \vec{t_1} \right)=0.$$
The last condition can be simplified, because $\nabla_n \vec{\sigma}$ is symmetric, to
$$ \nabla_n \vec{\sigma} \cdot \left(\vec{t_1} \otimes \vec{t_2}\right)=0. $$
Since $\vec{\sigma}$ has six components, the essential boundary conditions $\vec{\sigma} \vec{n} = \vec{0}$ along with these natural boundary conditions
present a total of six boundary conditions as needed.

To summarize, the eigenvalue problem developed earlier for two dimensions is extended in principle to three dimensions as follows:
\begin{eqnarray}
\nonumber
-\Delta \vec{\sigma} + \nabla_s \vec{\mu}  = \lambda \vec{\sigma} \,\, \text{ and } \,\, \mbox{div} \, \vec{\sigma} = \vec{0} & \text{ in }  \Omega,\\
\vec{\sigma n} = \vec{0} & \text{ on }  \partial \Omega, \nonumber\\
\nabla_n \vec{\sigma} \cdot (\vec{t_1} \otimes \vec{t_1}) = 0,  \, \, \, \nabla_n \vec{\sigma} \cdot (\vec{t_2} \otimes \vec{t_2}) = 0  \, \, \, \text{ and }  \, \, \, \nabla_n \vec{\sigma} \cdot (\vec{t_1} \otimes \vec{t_2}) = 0 & \text{ on }  \partial \Omega, \nonumber
\end{eqnarray}
for any two orthonormal unit vectors $\vec{t_1}$ and $\vec{t_2}$ tangential to the surface at the point of interest.

In the above, $\vec{\sigma}$ is a symmetric three dimensional second order tensor field and $\vec{\mu}$ is a three dimensional vector.

The proof of orthogonality of eigenfunctions, and the fact that they form a basis, proceeds along lines identical to the two dimensional case, and is omitted.

\section{Conclusions}
In this paper we set out to develop a sequence of stress fields which can serve as a basis for describing an arbitrary existing residual stress state in a given body.
Prior theoretical work on residual stresses has largely focused on specific mechanisms that generate such stresses, or occasionally discussed general aspects of such stresses.
In contrast, here we have proposed a specific, geometry-dependent, coordinate-system independent, extremization problem that leads to an eigenvalue problem whose spectrum provides such a basis. We have initially restricted the discussion to two dimensions, but shown later that the method can be extended to three dimensions. We have proved that the sequence of eigenfunctions indeed provides an orthonormal basis for the stress states under consideration. We have demonstrated some finite element solutions for such basis functions on three different domains, and then computed many basis functions for an annular domain using a semi-analytical approach. Finally, we have demonstrated that five different, rather arbitrary, residual stress states can indeed be approximated to arbitrary accuracy (in the $L^2$ norm, Eq.\ \ref{norm}) using our basis functions.

We note here that there are some philosophical similarities between our approach to constructing a basis for residual stress, and the study of the Stokes
operator \cite{temam} from incompressible fluid mechanics. However, our residual stresses are symmetric tensor fields, and our equilibrium equations are vector valued; while for the Stokes operator the velocities are vector fields, and the incompressibility implies a scalar constraint.

The importance of our work is twofold.

Academically speaking, we present a departure from the usual theoretical approach wherein examination of residual stress states is closely tied to their mechanical origins. Our approach recognizes that the basis must be generated afresh for every body geometry, but is otherwise free of the mechanical origins of the residual stresses. This is, in principle, like the construction of Fourier series as a basis for periodic functions with period $T >0$, independent of the physical origins of the periodicity; or the use of normal vibration modes as a basis to represent static deflections of a body under general loading. The basis we compute is a property of the body's shape and size, independent of its constitutive behavior.

In practical terms, we believe that our work opens the door to valuable new computations in industrial settings. For example, at the end of a metal forming calculation using nonlinear elastoplastic simulation, residual stress states in the unloaded body are often just displayed graphically. Now, the coefficients from an expansion using our basis can provide a useful new way of numerically describing those stress states. As another example, if the residual stress state in a body is experimentally determined at $N$ isolated points, there was so far no theoretically well-defined and mechanically consistent way to interpolate those stresses and make an assessment of possible residual stress states elsewhere in the body. The use of a basis, such as we have developed here, suggests a new research direction.

\section*{Funding sources}
This research did not receive any specific grant from funding agencies in the public, commercial, or not-for-profit sectors.

\section*{Acknowledgements}
We thank Anurag Gupta, Animesh Pandey, Ayan Roychowdhury, Sovan Das, and Jim Jenkins for technical discussions and encouragement. 

\section*{Conflict of interest}
The authors declare that they have no conflict of interest.

\begin{appendices}

\section{Proof of Eq. \ref{ip2}} \label{math_reductions}
First, we show that 
\begin{equation}
\label{ip1a}
\int_{\Omega}  \nabla_s \vec{\mu}  \cdot \vec{\sigma} \, dA = 0
\end{equation}
in Eq.\ \ref{ip}.
Note that
$$\nabla_s \vec{\mu}  \cdot \vec{\sigma} = \frac{\mu_{i,j}+\mu_{j,i}}{2} \, \sigma_{ij} = \mu_{i,j} \sigma_{ij} = \left ( \mu_i \sigma_{ij} \right )_{,j} - \mu_i \sigma_{ij,j} = \left ( \mu_i \sigma_{ij} \right )_{,j} $$
where we have used $\sigma_{ij}=\sigma_{ji}$ and $\sigma_{ij,j}=0$. Using the divergence theorem,
$$\int_{\Omega}  \left ( \mu_i \sigma_{ij} \right )_{,j} dA = 
\int_{\partial \Omega}  \mu_i \sigma_{ij} n_j ds = 0$$
because $\sigma_{ij} n_j = 0$ on $\partial \Omega$. Thus Eq.\ \ref{ip} becomes
\begin{equation}
\label{ip1} \int_{\Omega} \left ( - \Delta \vec{\phi}  - \lambda \vec{\phi} \right )  \cdot \vec{\sigma} dA = 0.
\end{equation}
Next, observe that
$$ \Delta \vec{\phi} \cdot \vec{\sigma} = \phi_{ij,kk} \sigma_{ij} = \left ( \phi_{ij,k} \sigma_{ij} \right )_{,k} - \phi_{ij,k} \sigma_{ij,k}.$$
In the right hand side above,
$$ \phi_{ij,k} \sigma_{ij,k} = \nabla \vec{\phi} \cdot \nabla \vec{\sigma},$$
while
$$ \int_{\Omega} \left ( \phi_{ij,k} \sigma_{ij} \right )_{,k} \, dA = \int_{\partial \Omega} \phi_{ij,k} \sigma_{ij} n_k \, ds.$$ 
Recalling Eq.\ \ref{q3q4} and the related discussion, symmetry of $\vec{\sigma}$ means
$\sigma_{ij} = \kappa(s) t_i t_j$ for some scalar $\kappa(s)$, and so
$$ \int_{\partial \Omega} \phi_{ij,k} \sigma_{ij} n_k \, ds = 0,$$
proving Eq. \ref{ip2}.

\section{Unit ball in $\mathcal{S}_{N\perp}$ contains an extremizer of $J_0$} \label{snperp}
The unit ball in $\mathcal{S}_{N\perp}$ is understood to be the set 
$$\mathcal{P}=\biggl\{\vec{\sigma} \left | \vec{\sigma} \in \mathcal{S}_{N\perp}, \hspace{1mm} \left(\int_{\Omega} \vec{\sigma} \cdot \vec{\sigma} \, dA\right)^{\frac{1}{2}}=1  \right. \biggr\}.$$
In this section, we show that $\mathcal{P}$ contains an extremizer of \\
 $\displaystyle J_0(\vec{\sigma})=\frac{1}{2}\int_{\Omega} \nabla \vec{\sigma} \cdot \nabla \vec{\sigma} \, dA$.
 
When we wish to include the elements in $\mathcal{S}_{N\perp}$ which have norm less than 1 as well, we will use the symbol $\bar{\mathcal{P}}$, as in 
$$\bar{\mathcal{P}}=\biggl\{\vec{\sigma} \left | \vec{\sigma} \in \mathcal{S}_{N\perp}, \hspace{1mm} \left(\int_{\Omega} \vec{\sigma} \cdot \vec{\sigma} \, dA\right)^{\frac{1}{2}} \leq 1 \right. \biggr\}. $$

Recall that $\mathcal{S}_{N\perp}$ is the orthogonal complement of $\mathcal{S}_N$ in $\mathcal{S}$. If $N$ is finite, $\mathcal{S}_{N\perp}$ is infinite dimensional. However, our arguments below only require the dimensionality of $\mathcal{S}_{N\perp}$ to be $\geq$ 2.

We note that the problems of minimizing $J_0(\vec{\sigma})$ and minimizing \\ $\hat{J}(\vec{\sigma})=(2J_0(\vec{\sigma}))^{\frac{1}{2}}=\displaystyle \left(\int_{\Omega} \nabla \vec{\sigma} \cdot \nabla \vec{\sigma} \, dA \right)^{\frac{1}{2}}$ are equivalent. The values of $\hat{J}$ evaluated in $\mathcal{P}$ have a greatest lower bound $\psi_0$. Thus, there exists a sequence $(\vec{\sigma}_n)$ in $\mathcal{P}$ such that
\begin{equation}
\nonumber
\displaystyle{\lim_{n \to \infty}} \hat{J}(\vec{\sigma}_n) = \psi_0.
\end{equation}
We must show that the limit of $(\vec{\sigma}_n)$ is in $\mathcal{P}$, and $\hat{J}$ evaluated at that limit is $\psi_0$.

We will use the $L^2$ and $H^1$ norms of a function $\vec{\sigma} \in \mathcal{S}$, as in  
$$\norm{\vec{\sigma}}_{L^2}=\left(\int_{\Omega} \vec{\sigma} \cdot \vec{\sigma} \, dA \right)^{\frac{1}{2}}, $$
$$\norm{\vec{\sigma}}_{H^1}=\left(\int_{\Omega} \vec{\sigma} \cdot \vec{\sigma} \, dA  +\int_{\Omega} \nabla \vec{\sigma} \cdot \nabla \vec{\sigma} \, dA \right)^{\frac{1}{2}}.$$

Our proof will proceed using the following steps. First we will show that the sequence $(\vec{\sigma}_n)$ is bounded in the $H^1$ norm, and thus has a subsequence that converges weakly in the $H^1$ norm, and strongly in the $L^2$ norm, to some $\vec{\sigma}_0$. We will then show that $\vec{\sigma}_0$ belongs to $\mathcal{P}$. Finally, we will show that although $\hat{J}$ is not continuous, it is lower semi-continuous, a property which implies that $\hat{J}$ achieves $\psi_0$ at $\vec{\sigma}_0$. 

\begin{proposition}
	$(\vec{\sigma}_n)$ is bounded in the $H^1$ norm.
\end{proposition}
\begin{proof}	
Since residual stresses have zero mean \cite{hoger1986}, by Poincar\'e's inequality \cite{Giovanni} there exists a positive real number $C$ that depends only on $\Omega$  such that
$$ \left(\int_{\Omega} \vec{\sigma} \cdot \vec{\sigma} \, dA \right)^{\frac{1}{2}} \leq  C \left(\int_{\Omega} \nabla \vec{\sigma} \cdot \nabla \vec{\sigma} \, dA \right)^{\frac{1}{2}} \,\, \forall \vec{\sigma}\in \mathcal{S}.$$ 
This implies that
$$ \norm{\vec{\sigma}}_{H^1} \leq \sqrt{C^2+1}\, \hat{J}(\vec{\sigma}).$$
By definition, all the elements in the sequence $(\vec{\sigma}_n) \in \mathcal{P}$ yield finite $\hat{J}$. We then conclude from the above equation that $(\vec{\sigma}_n)$ is bounded in the $H^1$ norm.
\end{proof}
    
\begin{proposition}
	$(\vec{\sigma}_n)$ has a subsequence that converges to some $\vec{\sigma}_0$ weakly in the $H^1$ norm, and strongly in the $L^2$ norm.
	%\footnote{Weak convergence can be thought of as convergence in an {\em average} sense. For example, in $0\leq x \leq \pi$, $\sin{n x}$ converges in the $L^2$ norm to zero not absolutely (strongly), but in an average sense (weakly), as $n$ goes to infinity.} . 
\end{proposition}	
\begin{proof}
	It is well known in the theory of functional analysis that $H^1$ (set of $\vec{\sigma}$ with finite $H^1$ norm) is a Banach space. Every bounded sequence in a Banach space has a weakly convergent subsequence (see corollary A.60, page 506 of \cite{Giovanni}).  It follows that there is a subsequence $(\vec{\sigma}_{n_k})$ of $(\vec{\sigma}_{n})$ that converges weakly to some $\vec{\sigma}_0 \in H^1.$ By the Rellich-Kondrachov theorem \cite{Giovanni}, $H^1$ is compactly embedded in $L^2$ (set of $\vec{\sigma}$ with finite $L^2$ norm), and therefore $(\vec{\sigma}_{n_k})$ converges strongly to $\vec{\sigma}_0$ in the $L^2$ norm (e.g., see exercise 3.5, page 80 of \cite{brezis}).
	%\footnote{Intuitively, for a sequence in a bounded domain to be bounded in $L^2$ and still not have any subsequence that converges strongly in the $L^2$ norm, it must be oscillating rapidly in the limit, making its $H^1$ norm unbounded. But both $L^2$ and $H^1$ norm of $(\vec{\sigma}_n)$ and its limit are bounded. So, $(\vec{\sigma}_n)$ must have a subsequence that converges strongly in $L^2$.}.
\end{proof}	

\begin{proposition}
	$\vec{\sigma}_0 \in \mathcal{P}$.
\end{proposition}	
\begin{proof}
	Recall that $\mathcal{P}$ consists of elements $\vec{\sigma}$ that 
	\begin{enumerate}[(i)]
		\item are divergence-free,
		\item are traction-free,
		\item have $\displaystyle\int_{\Omega} \nabla \vec{\sigma} \cdot \nabla \vec{\sigma} \, dA < \infty,$
		\item are orthogonal to $\mathcal{S}_N$, and
		\item satisfy $\norm{\vec{\sigma}}_{L^2}=1.$
	\end{enumerate}
	
	We now show that $\vec{\sigma}_0$ satisfies each of the above conditions (i) through (v).
  
\begin{enumerate}[(i)]
		\item \underline{$\mbox{div}\,\vec{\sigma}_0 = \vec{0}$}: 
		
		Since $\vec{\sigma}_0$ is in $H^1$, $\mbox{div}\,\vec{\sigma}_0$ is a vector field in $L^2$. Let $\vec{\gamma}           $ be an arbitrary smooth vector field compactly supported over $\Omega$.  Consider the inner product of $\mbox{div}\,\vec{\sigma}_0$ with $\vec{\gamma}$. Using integration by parts followed by H\"older's inequality, we have
		$$ 
		\int_{\Omega} \mbox{div}\,\vec{\sigma}_0 \cdot \vec{\gamma}  \,dA =\int_{\Omega} \mbox{div} \,(\vec{\sigma}_0-\vec{\sigma}_{n_k}) \cdot \vec{\gamma} \, dA \, + \int_{\Omega} \mbox{div}\,\vec{\sigma}_{n_k} \cdot \vec{\gamma}  \,dA $$
		$$=-\int_{\Omega} (\vec{\sigma}_0-\vec{\sigma}_{n_k}) \cdot \nabla \vec{\gamma} \, dA
		\leq \norm{\vec{\sigma}_0-\vec{\sigma}_{n_k}}_{L^2} \norm{\nabla \vec{\gamma}}_{L^2}.
		$$  
		Since $\vec{\gamma}$ is smooth, $\norm{\nabla \vec{\gamma}}_{L^2}$ is finite, and the right-most expression in the above equation goes to zero. So, we have
		$ \displaystyle \int_{\Omega} \mbox{div} \, \vec{\sigma}_0 \cdot \vec{\gamma} \, dA \leq 0.$ Choosing $-\vec{\gamma}$ in place of $\vec{\gamma}$ gives 
		$\displaystyle \int_{\Omega} \mbox{div} \, \vec{\sigma}_0 \cdot \vec{\gamma} \, dA \geq 0. $	
		We conclude that
		$$ \int_{\Omega} \mbox{div} \, \vec{\sigma}_0 \cdot \vec{\gamma} \, dA =0. $$
		
		Since $\vec{\gamma}$ is arbitrary, and smooth compactly supported functions are dense in $L^{2}$ \cite{brezis}, we conclude that
		$$\mbox{div} \, \vec{\sigma}_0 = \vec{0}.$$
		\item \underline{$\vec{\sigma}_0 \vec{n}=\vec{0}$}:
		
		 For an arbitrary smooth vector field $\vec{\chi}$, using integration by parts, we have
		$$0=\int_{\Omega} \mbox{div} \,(\vec{\sigma}_0-\vec{\sigma}_{n_k}) \cdot \vec{\chi} \, dA = \int_{\partial \Omega} \left\{(\vec{\sigma}_0-\vec{\sigma}_{n_k}) \vec{n}\right\} \cdot \vec{\chi} \, ds\, - \int_{\Omega} (\vec{\sigma}_0-\vec{\sigma}_{n_k}) \cdot \nabla \vec{\chi} \, dA,$$
		or, since $\vec{\sigma}_{n_k} \vec{n}=\vec{0}$, 
		$$ \int_{\partial \Omega} (\vec{\sigma}_0 \vec{n}) \cdot \vec{\chi} \, ds=\int_{\Omega} (\vec{\sigma}_0-\vec{\sigma}_{n_k}) \cdot \nabla \vec{\chi} \, dA.$$
		Again using H\"older's inequality, we obtain that 
		$$ \vec{\sigma}_0 \vec{n}=\vec{0}.$$
		\item \underline{$\displaystyle \int_{\Omega} \nabla \vec{\sigma} \cdot \nabla \vec{\sigma} \, dA < \infty$}: 
		
		This is obvious since $\vec{\sigma}_0$ belongs to $H^1$ (proposition 2).
		\item \underline{$\vec{\sigma}_0$ is orthogonal to $\mathcal{S}_N$}:

		Since $\vec{\sigma}_{n_k}$ is orthogonal to $\mathcal{S}_N$ for all $n_k$, and inner product is a continuous function \cite{Bossavit}, we conclude that $\vec{\sigma}_0$ is orthogonal to  $\mathcal{S}_N$.

		\item \underline{$\norm{\vec{\sigma}_0}_{L^2}=1$}:
		
		Again, since $\norm{\vec{\sigma}_{n_k}}_{L^2}=1$ for all $n_k$, and norm is a continuous function \cite{Bossavit}, we conclude that $\norm{\vec{\sigma}_0}_{L^2}=1$.
\end{enumerate} 
\end{proof}	
\begin{remark}
$\vec{\sigma}_{n_k}$ converges to $\vec{\sigma}_0$ strongly in the $L^2$ norm, and the corresponding $\hat{J}$ values converge to $\psi_0$. However, it is not clear if $\hat{J}(\vec{\sigma}_0)=\psi_0$, since as a function from $\mathcal{S} \subset L^2$ to $\mathbb{R}$, $\hat{J}$ is not continuous. In the following arguments, we show that $\hat{J}$ satisfies a weaker but sufficient condition, that of lower semi-continuity. 
\end{remark}    
\begin{definition}
	A functional $f$ is strong (respectively, weak) lower semi-continuous with respect to a norm if it satisfies 
$$ f(\vec{x}_0)\leq \displaystyle{\lim_{m \to \infty}} f(\vec{x}_m)$$
whenever a sequence $(\vec{x}_m)$ converges strongly (respectively, weakly) to $\vec{x}_0$ in that norm \cite{ekland}. 
\end{definition}
\begin{remark}
	Weak and strong lower semi-continuity of a functional are related as follows. In general, weak lower semi-continuity in a norm implies strong lower semi-continuity in that norm. The converse is not true. However, if the functional is strong lower semi-continuous and convex, and is defined on a convex set, then it is weak lower semi-continuous \cite{ekland}.  
\end{remark} 
\begin{proposition}
	$\hat{J}$ is strong lower semi-continuous in $H^1$.
\end{proposition}	     
\begin{proof}
	We first show that the quantity defined as $\norm{\vec{\sigma}}_{\hat{J}}= \hat{J}(\vec{\sigma})$ is a norm over set $\mathcal{S}$. Since residual stresses have zero mean, $\norm{\vec{\sigma}}_{\hat{J}}$ is zero only when $\vec{\sigma}$ is zero. Also, $\norm{\alpha \vec{\sigma}}_{\hat{J}}=|\alpha| \norm{\vec{\sigma}}_{\hat{J}}$ for a real number $\alpha$. Finally, using H\"older's inequality, it can easily be shown that $\hat{J}(\vec{\sigma})$ satisfies the triangle inequality. So, $\norm{\vec{\sigma}}_{\hat{J}}$ is a norm. 
	
	Next, we note that
	$$ 0\leq \left(\int_{\Omega} \nabla \vec{\sigma} \cdot \nabla \vec{\sigma} \, dA\right)^{\frac{1}{2}} \leq \left(\int_{\Omega} \vec{\sigma} \cdot \vec{\sigma} \, dA + \int_{\Omega} \nabla \vec{\sigma} \cdot \nabla \vec{\sigma} \, dA\right)^{\frac{1}{2}},$$
	or
	$$0\leq\norm{\vec{\sigma}}_{\hat{J}} \leq \norm{\vec{\sigma}}_{H^1}.$$
	 
	Therefore, if a sequence $(\tilde{\vec{\sigma}}_m)$ converges strongly to some $\tilde{\vec{\sigma}}_0$ in the $H^1$ norm, i.e. $\displaystyle{\lim_{m \to \infty}} \norm{\tilde{\vec{\sigma}}_0-\tilde{\vec{\sigma}}_m}_{H^1} = 0$, it follows from the above that $\displaystyle{\lim_{m \to \infty}} \norm{\tilde{\vec{\sigma}}_0-\tilde{\vec{\sigma}}_m}_{\hat{J}} = 0$. This implies that  
	$$ \displaystyle{\lim_{m \to \infty}} \hat{J} (\tilde{\vec{\sigma}}_m)=\hat{J}(\tilde{\vec{\sigma}}_0).$$
	
	Hence, $\hat{J}$ is strong lower semi-continuous in $H^1$.
\end{proof}
 
\begin{proposition}
	$\hat{J}$ is weak lower semi-continuous in $H^1$.
\end{proposition}
\begin{proof}
	Since all norms are convex, $\hat{J}$ is a convex functional. The set $\bar{\mathcal{P}}$ defined earlier is convex. Including proposition 4, we conclude that $\hat{J}$ is weak lower semi-continuous in $H^1$ over $\bar{\mathcal{P}}$.
\end{proof}	
\begin{proposition}
	$\hat{J} (\vec{\sigma}_0)=\psi_0.$
\end{proposition}	   
\begin{proof}
Since $\hat{J}$ is weak lower semi-continuous in $H^1$, and $(\vec{\sigma}_{n_k}) \in \bar{\mathcal{P}}$ converges weakly to $\vec{\sigma}_0 \in \bar{\mathcal{P}}$ in the $H^1$ norm,
$$ \hat{J} (\vec{\sigma}_0) \leq \displaystyle{\lim_{n_k \to \infty}} \hat{J}(\vec{\sigma}_{n_k})=\psi_0.$$
But since $\psi_0$ is the greatest lower bound of $\hat{J}$ over $\mathcal{P}$, and $\vec{\sigma}_0$ belongs to $\mathcal{P}$, we have 
$$\psi_0 \leq \hat{J} (\vec{\sigma}_0).$$ 
Hence, 
$$\hat{J} (\vec{\sigma}_0)=\psi_0.$$
\end{proof}

\begin{remark}
	We have now proved that $\vec{\sigma}_0$ is in $\mathcal{P}$, minimizes $\hat{J}$, and hence minimizes $J_0$. 
\end{remark} 

\section{Proof that the eigenfunctions form a basis for $\mathcal{S}$} \label{infinity}
Assume that the span of the infinitely many eigenfunctions is a subspace $\mathcal{S}_{\infty}$ which is a proper subspace of $\mathcal{S}$.
It is not clear what the dimension of its orthogonal complement $\mathcal{S}_{\infty \perp}$ is.

If the dimension is 2 or more, then it contains infinitely many elements of unit norm, and the arguments used in the first part of section \ref{basis_sec} can be applied and the same contradiction is
obtained.

If the dimension of $\mathcal{S}_{\infty \perp}$ is 1, then we have a unique (up to a scalar multiple) $\vec{\tau}$ which lies in $\mathcal{S}_{\infty \perp}$. Normalizing that $\vec{\tau}$,
we find that we cannot take variations of it while keeping it inside $\mathcal{S}_{\infty \perp}$. This precludes variational equations, and a different argument is easier.

The eigenfunctions $\{ \vec{\phi}_k \}$ along with $\vec{\tau}$ form a basis for $\mathcal{S}$. The issue is solely whether $\vec{\tau}$, too, is an eigenfunction.

Let us now consider a different extremization
problem, namely: find a $\vec{\sigma}$ in $\mathcal{S}$ that extremizes $\displaystyle\frac{1}{2}\int_{\Omega} \nabla \vec{\sigma} \cdot \nabla \vec{\sigma} \, dA$ subject to the
condition $\displaystyle\int_{\Omega} \vec{\sigma} \cdot \vec{\tau} \, dA = 1$. We make no assumptions about how the extremizing $\vec{\sigma}$ might be related to the eigenfunctions $\{ \vec{\phi}_k \}$.

We approach this new problem in two ways: (i) using the calculus of variations, and (ii) directly.

Using the calculus of variations, we obtain:
\begin{equation}
\begin{array}{cccl}
-\Delta \vec{\sigma} + \nabla_s \vec{\mu}  = \lambda \vec{\tau}  & \text{ and } & \mbox{div} \, \vec{\sigma} = \vec{0} & \text{ in} \hspace{1mm} \Omega,\\
\vec{\sigma n} = \vec{0} & \text{ and } & \nabla_n \vec{\sigma} \cdot (\vec{t} \otimes \vec{t}) = 0 & \text{ on} \hspace{1mm} \partial \Omega.
\end{array}
\label{dd}
\end{equation}

Using a direct approach, we can assume a solution of the form
$$\vec{\sigma} = a_0 \vec{\tau} + \sum_{k=1}^{\infty} a_k \vec{\phi}_k.$$
The above representation contains the solution because by assumption we have a basis. Directly substituting into
$$\frac{1}{2}\int_{\Omega} \nabla \vec{\sigma} \cdot \nabla \vec{\sigma}  \, dA$$
and using orthogonality, we find that we are extremizing
$$\frac{1}{2} \left ( a_0^2 \int_{\Omega} \nabla \vec{\tau} \cdot \nabla \vec{\tau} \, dA + \sum_{k=1}^{\infty} a_k^2 \lambda_k \right ),$$
subject to the constraint
$$\int_{\Omega} \vec{\sigma} \cdot \vec{\tau} \, dA = a_0 =1.$$
In the above, the $\lambda_k$'s are the eigenvalues already found.
The minimizer is obvious: $a_0 = 1$, and $a_k = 0$ for $k=1,2,3,\cdots$. It follows that the extremizing $\vec{\sigma}$ is exactly $\vec{\tau}$.
Therefore, $\vec{\sigma} = \vec{\tau}$ must also satisfy Eqs.\ \ref{dd}, obtained using the calculus of variations. We conclude that $\vec{\tau}$ is an eigenfunction
after all.

\section{Details of the finite element method} \label{FEM}
To solve the eigenvalue problem in Eq.\ \ref{three_eqns} using FEM, we first express it in Cartesian coordinates, so that $$-\Delta \vec{\sigma} + \nabla_s \vec{\mu} = \lambda \vec{\sigma}$$ becomes
\begin{equation}
\label{fem}
\begin{split}
-\frac{\partial^2 \sigma_{xx}}{\partial x^2} - \frac{\partial^2 \sigma_{xx}}{\partial y^2} + \frac{\partial \mu_x}{\partial x} = \lambda \sigma_{xx}, \\
-\frac{\partial^2 \sigma_{yy}}{\partial x^2} - \frac{\partial^2 \sigma_{yy}}{\partial y^2} + \frac{\partial \mu_y}{\partial y} = \lambda \sigma_{yy},  \\ 
-\frac{\partial^2 \sigma_{xy}}{\partial x^2} - \frac{\partial^2 \sigma_{xy}}{\partial y^2} + \frac{1}{2} \left(\frac{\partial \mu_x}{\partial y}+\frac{\partial \mu_y}{\partial x}\right) = \lambda \sigma_{xy}.
\end{split}
\end{equation}
Equilibrium,
$$ \mbox{div} \, \vec{\sigma} = \vec{0},$$ 
becomes
\begin{equation}
\label{eqb}
\begin{split}
\frac{\partial \sigma_{xx}}{\partial x} + \frac{\partial \sigma_{xy}}{\partial y} =0, \\
\frac{\partial \sigma_{xy}}{\partial x} + \frac{\partial \sigma_{yy}}{\partial y} =0. 
\end{split}
\end{equation}
The traction-free boundary condition 
$$ \vec{\sigma} \vec{n} = \vec{0}$$ 
gives
\begin{equation}
\label{bcxy}
\begin{split}
\sigma_{xx}n_x + \sigma_{xy} n_y=0, \\
\sigma_{xy}n_x + \sigma_{yy} n_y=0. 
\end{split}
\end{equation}
The final, and natural, boundary condition 
$$ \nabla_n \vec{\sigma} \cdot (\vec{t} \otimes \vec{t}) = 0$$ 
becomes
\begin{equation}
\label{nbc}
\frac{\partial \sigma_{xx}}{\partial x} n_x n_y^2 + \frac{\partial \sigma_{xx}}{\partial y} n_y^3 + \frac{\partial \sigma_{yy}}{\partial x} n_x^3+ \frac{\partial \sigma_{yy}}{\partial y} n_x^2 n_y -2\frac{\partial \sigma_{xy}}{\partial x} n_x^2 n_y - 2 \frac{\partial \sigma_{xy}}{\partial y} n_x n_y^2=0.
\end{equation}

We discretise the domain with a mesh containing `$e$' eight noded quadrilateral \textit{serendipity} elements and `$n$' nodes. Figure \ref{mesh} shows a sample mesh for $e=4$ and $n=20$ for a square domain. We use the FEM software package Abaqus to generate the mesh.

We use piecewise cubic shape functions for the stress components. The shape function that takes the value 1 at node $p$, and zero at all other nodes, is denoted as $N_{p}$. Each such shape function is cubic within individual elements, continuous on element edges, and looks like a tent peaking at node $p$.

We use piecewise constant shape functions for components of the Lagrange multiplier vector field $\vec{\mu}$. The piecewise constant shape function that is 1 on element $q$, and zero on all the other elements, is denoted as $M_q$.  
\begin{figure}[t!]
	\centering
	\includegraphics [scale=0.3]{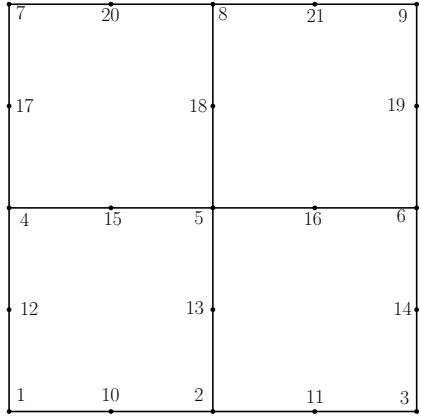}
	\caption{A sample 2 $\times$ 2 mesh of serendipity elements}
	\label{mesh}
\end{figure}

The discretised dependent variables are written as
\begin{equation}
\label{expansions}
\begin{split}
\sigma_{xx}=\sigma_{xx_{1}}N_{1}+\sigma_{xx_{2}}N_{2}+...+\sigma_{xx_{n}}N_{n}, \\
\sigma_{yy}=\sigma_{yy_{1}}N_{1}+\sigma_{yy_{2}}N_{2}+...+\sigma_{yy_{n}}N_{n}, \\
\sigma_{xy}=\sigma_{xy_{1}}N_{1}+\sigma_{xy_{2}}N_{2}+...+\sigma_{xy_{n}}N_{n}, \\
\mu_{x}=\mu_{x_{1}}M_{1}+\mu_{x_{2}}M_{2}+...+\mu_{x_{e}}M_{e}, \\ 
\mu_{y}=\mu_{y_{1}}M_{1}+\mu_{y_{2}}M_{2}+...+\mu_{y_{e}}M_{e},
\end{split}
\end{equation}
where $\sigma_{xx_{p}}$ denotes the value of the discretized $\sigma_{xx}$ component at the $p^{\, th}$ node (likewise for $\sigma_{yy}$ and $\sigma_{xy}$); and where $\mu_{x_q}$ denotes the value of the discretized $\mu_x$ over element $q$ (likewise for $\mu_y$). 

We arrange the $3n+2e$ unknowns in a column vector $c$ as follows:
\drop{
\begin{align}
\label{cdef}
c &= \begin{Bmatrix}
\sigma_{xx_1} \\
\vdots \\
\sigma_{xx_n}\\
\sigma_{yy_1} \\
\vdots \\
\sigma_{yy_n}\\
\sigma_{xy_1} \\
\vdots \\
\sigma_{xy_n}\\
\mu_{x_1} \\
\vdots \\
\mu_{x_e}\\
\mu_{y_1} \\
\vdots \\
\mu_{y_e}\\
\end{Bmatrix}.
\end{align}
}
\begin{equation}
\label{cdef}
c=\left\{\sigma_{xx_1}\,.\,.\,.\,\sigma_{xx_n}\,\,\sigma_{yy_1}\,.\,.\,.\,\sigma_{yy_n}\,\,\sigma_{xy_1}\,.\,.\,.\,\sigma_{xy_n}\,\,\mu_{x_1}\,.\,.\,.\,\mu_{x_e}\,\,\mu_{y_1}\,.\,.\,.\,\mu_{y_e}\right\}^T,
\end{equation}
$T$ denoting transpose.
We need $3n+2e$ equations. For the first $n$ equations, we take the inner product of first of Eqs.\ \ref{fem} with $N_1$ through $N_n$. For instance, the first such resulting equation is:
\begin{equation}
\nonumber
\int_{\Omega}\left(- \Delta \sigma_{xx}+\mu_{x,x}-\lambda \sigma_{xx} \right)N_{1}\,dA=0.
\end{equation}
Using integration by parts, we obtain
\begin{equation}
\nonumber
\begin{array}{rr}
\displaystyle \int_{\partial\Omega} \left(-\frac{\partial \sigma_{xx}}{\partial x}n_{x}-\frac{\partial \sigma_{xx}}{\partial y}n_{y}+\mu_{x} n_{x}\right)N_{1}\,ds-\int_{\Omega}\left(-\frac{\partial \sigma_{xx}}{\partial x}\frac{\partial N_{1}}{\partial x}-\frac{\partial \sigma_{xx}}{\partial y}\frac{\partial N_{1}}{\partial y}+\mu_{x} \frac{\partial N_{1}}{\partial x}\right)dA \\
= \displaystyle \lambda \int_{\Omega}\sigma_{xx}N_{1}\,dA.
\end{array}
\end{equation}
We substitute from Eqs.\ \ref{expansions} to obtain
\begin{equation}
\begin{split}
\sum_{r=1}^{n}\sigma_{xx_{r}}\left\{\int_{\partial\Omega}-\left(\frac{\partial N_{r}}{\partial x}n_{x}+\frac{\partial N_{r}}{\partial y}n_{y}\right)N_{1}\,ds+\int_{\Omega}\left(\frac{\partial N_{r}}{\partial x}\frac{\partial N_{1}}{\partial x}+\frac{\partial N_{r}}{\partial y}\frac{\partial N_{1}}{\partial y}\right)dA\right\}\\
+\sum_{s=1}^{e}\mu_{x_{s}}\left(\int_{\partial\Omega}N_{1}n_{x}\,ds-\int_{\Omega}\frac{\partial N_{1}}{\partial x}\,dA\right)=\lambda\sum_{r=1}^{n}\sigma_{xx_{r}}\int_{\Omega}N_{r}N_{1}\,dA.
\end{split}
\label{subs}
\end{equation}

The various integrals in the above equation are all meaningful because each $N_r$ as well as its gradient $\nabla N_r$ are bounded everywhere
in the domain, including on the boundary; and the shape functions used for $\vec{\mu}$ are piecewise constant\footnote{%
	The stress components are in the Sobolev space $H^1(\Omega)$; their restrictions to the boundary are
	in $H^1(\partial \Omega)$; the $\mu$-components are in the Hilbert space $L^2(\Omega)$; and their restrictions to the boundary
	are in $L^2(\partial \Omega)$.}.
Additionally, we note that the boundary integrals above remain continuous even if $n_x$ and $n_y$ have a finite number of discontinuities,
i.e., the domain can have a finite number of corners.

Equation \ref{subs} (recall Eq.\ \ref{cdef}) can be written compactly as
\begin{equation}
\nonumber
a_{1} c=\lambda a_{2} c,
\end{equation}
where $a_1$ and $a_2$ are row vectors of dimensions $1 \times (3n+2e)$. We obtain $n-1$ more equations by taking the inner product of the first of Eqs.\ \eqref{fem} with $N_{2}$ through $N_{n}$.

Similarly, we obtain $2n$ more equations by taking the inner product of second and third of Eqs.\ \ref{fem} with $N_1$ through $N_n$. 

Finally, we obtain the remaining $2e$ equations by taking the inner product of both of Eqs.\ \ref{eqb} with each of
$M_1$ through $M_e$.
It can be verified easily, as for Eq.\ \ref{subs}, that all integrals in those equations are well behaved.

The complete set of $3n+2e$ equations can be written in a compact form as follows:
\begin{equation}
\label{matrixeq2}
A_1 c=\lambda A_2 c,
\end{equation} 
where $A_1$ and $A_2$ are square matrices of dimensions $(3n+2e) \times (3n+2e)$.
We have not imposed the boundary conditions (Eqs.\ \ref{bcxy} and \ref{nbc}) yet. We have enforced these in the weak form as well (in an integral sense, on the domain boundary; details omitted). If there are $b$ nodes on the boundary, there are $3b$ conditions to be imposed. 
The boundary conditions can be expressed in the form $Bc=0$, where $B$ is a $3b \times (3n+2e)$ matrix.

This means that the vector of unknowns is $3n+2e-3b$ dimensional.
For problems of moderate size, such as we solve here, it is conceptually simplest to compute a matrix $Q$ whose columns span the
subspace orthogonal to the rows of $B$. 
Then, Eq.\ \eqref{matrixeq2} along with boundary conditions can be reduced to an equation of the form
\begin{equation}
\label{matrixeqreduced2}
\tilde{{A_{1}}}\tilde{{c}}=\lambda\tilde{{A_{2}}}\tilde{{c}},
\end{equation}
where $c=Q \tilde{{c}}$, $\tilde{A_1} = Q^T A_1 Q$ and $\tilde{A_2} = Q^T A_2 Q$.
Equation \eqref{matrixeqreduced2} is an eigenvalue problem. One last point is that, because of the constraints in the problem, 
several eigenvalues are infinite. So we solve Eq.\ \ref{matrixeqreduced2} in the form
$$\tilde{{A_{2}}}\tilde{{c}}=\frac{1}{\lambda} \tilde{{A_{1}}}\tilde{{c}},$$
select the {\em largest} eigenvalues $1/\lambda$, and take their reciprocals.
Finally, we arrange the eigenvectors (eigenfunctions) in order of increasing $\lambda$. 

We have computed the eigenfunctions using the above formulation for three domains: an annular domain, a square domain, and an arbitrarily shaped domain (see figures \ref{annular_modes}, \ref{square_modes} and \ref{arbitrary_shapes}). We have performed convergence tests by refining the mesh, and displayed results in the main paper using a level of refinement at which the eigenvalues varied within tiny fractions of one percent. For instance, the domain corresponding to figure \ref{square_modes} was discretized using a mesh of $2500$ elements.     

Our numerical results indicate that our formulation is stable. However, we have not formally verified the well known inf-sup condition (also known as the Ladyzhenskaya-Babu\v ska-Brezzi condition) for our mixed finite element formulation. We refer the interested reader to Bathe's work \cite{Bathe,Bathe2,Bathe3} and the references therein (also see \cite{Falk} for mixed finite element formulations in linear elasticity). Here, we offer the following positive and constructive points to demonstrate the correctness of our finite element results.
\begin{enumerate}
	\item First, since our problem is similar to the Stokes problem, we observe on page 329, Table 4.8 of {\em Finite element procedures} by Bathe \cite{Bathe}, that the 8/1 element (the eight-noded quadrilateral serendipity element with piecewise constant pressure, which is what we have used) is stable for the Stokes problem. This does not {\em guarantee} that it will be stable for our problem, but it is indicative, and our results have not shown instabilities.

\item Second, the qualitative consequence of instability is the appearance of spurious checkerboard type patterns in the solution. In many solutions, at different mesh refinements, for different domain shapes, we have {\em not seen} such checkerboard patterns with our 8/1 element.

\item Third, with other elements, which are unsuitable, we did indeed obtain checkerboard patterns. Specifically, we did so with 4/1 elements, consistent with Table 4.8 in \cite{Bathe}

Conversely, with 9/3 and 9/4--c, two other suitable elements mentioned in the table, we obtained similar results as with the 8/1 element, with no spurious modes. 

\item Fourth, for an annular domain, we have used an independent semi-analytical method developed in section 6.2, where numerical ODE solution is used. The results match our FE solutions near-perfectly. This tells us that our FE solutions are not only stable (no checkerboard) but also accurate.

\item Finally, for the square domain shown in figure \ref{square_modes}, we demonstrate convergence numerically. We consider four meshes: $5 \times 5$, $10 \times 10$, $20 \times 20$ and $40 \times 40$ elements. The first ten eigenvalues from these four meshes are plotted in figure \ref{refine}. Convergence is clear. 
	
	\begin{figure}[h!]
		\centering
		\includegraphics [scale=0.6]{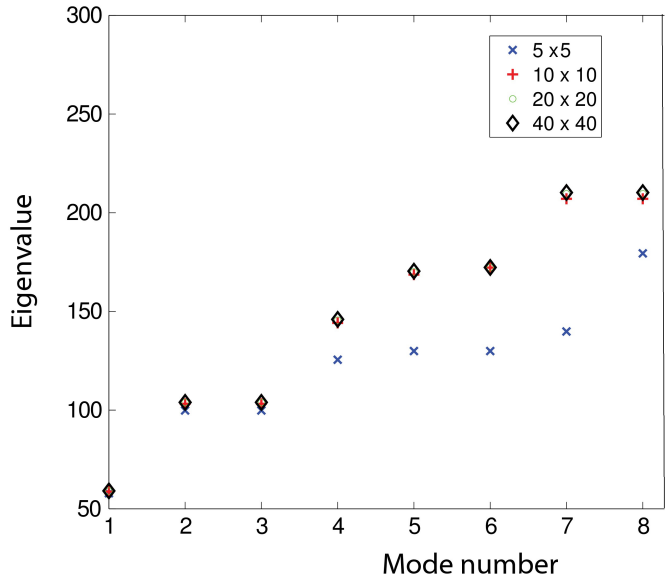}
		\caption{Convergence of eigenvalues upon refinement.}
		\label{refine}
	\end{figure}
	
\end{enumerate}

To close this section, we report the time required to compute the eigenfunctions using our own code in Matlab, on a personal computer with $8^{\mbox{th}}$ generation i5 processor. Computation of the first 100 eigenfunctions on a square domain discretized with uniform meshes of $5 \times 5$, $10 \times 10$, $20 \times 20$, $40 \times 40$, $80 \times 80$ and $160 \times 160$ elements takes 0.2, 0.6, 1.6, 10, 70 and 850 seconds respectively. The computation times are plotted in figure \ref{comptime} on a log-log scale.
      \begin{figure}[h!]
		\centering
		\includegraphics [scale=0.62]{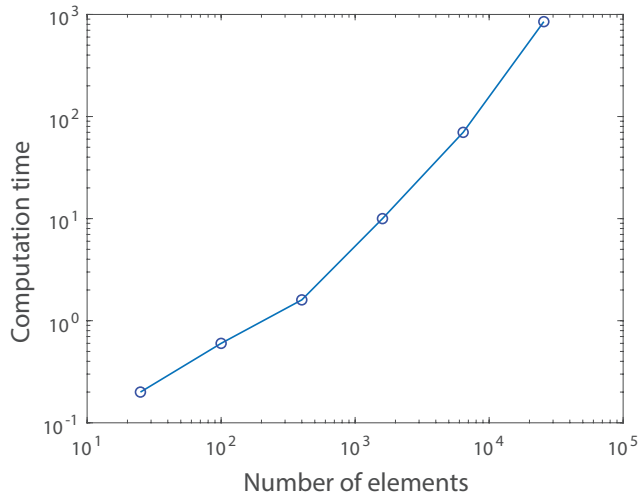}
		\caption{Time (in seconds) for computation of first one hundred eigenfunctions for different mesh refinements (log-log scale). The data points are joined by straight lines for visibility alone.}
		\label{comptime}
	\end{figure}

\section{Fitting of a residual stress field obtained from a metal forming simulation in Abaqus} \label{rolling}
In this section, we describe the process of generating a residual stress using the finite element software package Abaqus by simulating the 2-D metal forming process of rolling, and fit the residual stress field using our eigenfunctions computed on the same (final) mesh with our finite element code as described above.
 
\begin{subappendices}
\subsection{Details of the rolling simulation in Abaqus}
The schematic of the set-up is shown in figure \ref{schem}. The simulation is carried out quasi-statically, in the implicit analysis mode of Abaqus. The general description of the simulation is as follows: the workpiece is first nudged to the right using a rigid punch moving with a constant velocity, until the former comes in contact with the rotating rigid rollers. The friction between the workpiece and the rollers pulls the workpiece away from the rigid punch, and the formed workpiece is then extruded at the other end. 
\begin{figure}[!h]
	\centering
	\includegraphics[scale=0.3]{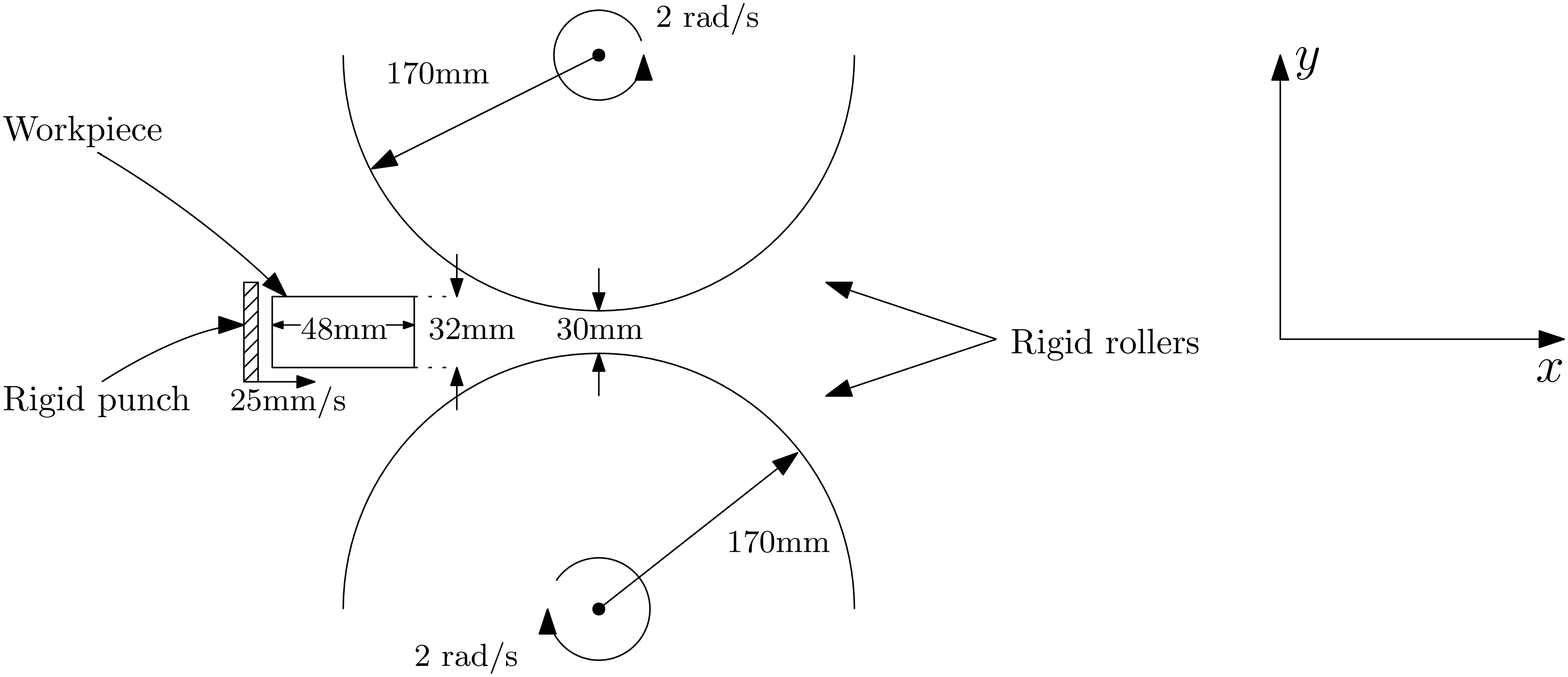}
	\caption{Schematic of the rolling simulation (figure not to scale).}
	\label{schem}
\end{figure}

The geometric, material and contact details are as follows. The workpice is 48 mm long and 32 mm wide, and is made of an isotropic elasto-plastic material with linear strain hardening. Its Young's modulus is 210 GPa, Poisson's ratio is 0.3, yield stress is 400 MPa, and slope of the hardening curve is such that the equivalent plastic strain is 10 when the von Mises stress is 6000 MPa. Since the process is quasi-static, density of the workpiece is not required. The punch is rigid. It moves with a velocity of 25 mm/s to right. Both rollers are rigid and each has a radius of 170 mm. They rotate at 2 rad/s in the directions indicated in figure \ref{schem}. The minimum gap between the rollers is 30 mm, so that the width of the formed workpiece is reduced by 2 mm in the process.   A `hard' normal contact is assumed between the punch and the workpiece, as well as the workpiece and the rollers. `Penalty' friction with a coefficient of 0.3 is assumed in each of these contacts. 

The mesh details are as follows. The rollers and punch are meshed with `discrete rigid' and `analytical rigid' line elements, respectively.  The workpiece is meshed with 20184 plane strain four-noded quadrilateral elements of size 0.275 mm and aspect ratio 1. Mesh convergence tests are performed by comparing the nodal values of different stress components along material lines for different element sizes, based on which we conclude that an element size of 0.275 mm provides a converged solution. 

The simulation is quasi-static, and is carried out in an implicit time step of size 6 seconds, with minimum increment size of $10^{-9}$ seconds, and initial increment of size $10^{-3}$ seconds. The mid-line ($y=0$) running across the length of the workpiece is constrained to not move in the $y$ direction by using rollers. This ensures that the normal (respectively, shear) stress components are symmetric (respectively, anti-symmetric) with respect to $y=0$. 

Readers can access the input file of this Abaqus simulation here: \\ {\color{blue}
\href{https://tinyurl.com/wefcwps}{https://tinyurl.com/wefcwps}}.
\subsection{Fitting results}
\begin{figure}[h!]
	\centering
	\includegraphics[scale=0.6]{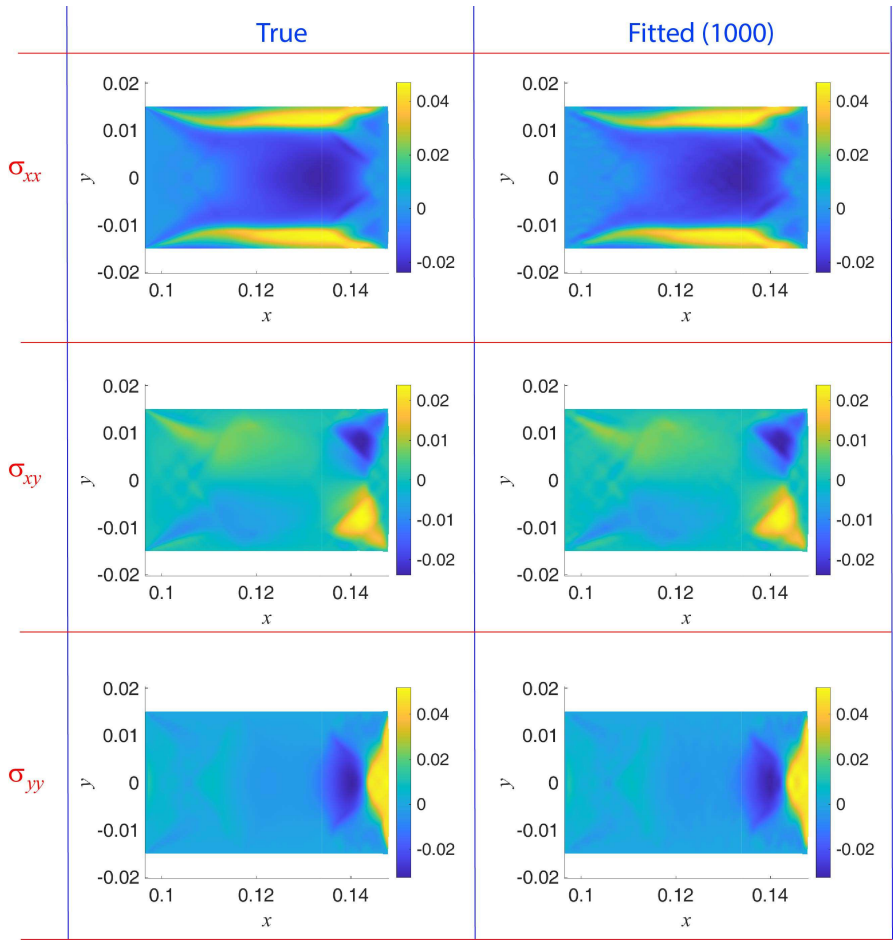}
	\caption{True and fitted (using 1000 eigenfunctions) stress components for metal forming example (in GPa).}
	\label{plots}
\end{figure}
We show the stress components obtained from the Abaqus simulation in the left column of figure \ref{plots}. This stress field is fitted using the first 1000 eigenfunctions computed over the same (deformed) mesh as obtained from the simulation, using the procedure described in appendix D. The fitted components are shown in the right column of figure \ref{plots}. We observe that the fit is good. We also plot the fitted components using 10, 60 and 102 eigenfunctions respectively in figure \ref{moreplots} to indicate how the fits get progressively better with incorporation of more eigenfunctions. Next, we plot the squared relative error measure $E_N$, described in section 7, versus $N$ in figure \ref{error}. Convergence is like $N^{-1}$ for large $N$, with $E_{318}<0.01$. Finally, we plot the time required for computation of the first 1000 eigenfunctions for different refinements of the mesh used for the rolling simulation in figure \ref{comptimeroll}. The coarser meshes used for this plot were obtained from different simulations done to study mesh convergence: the stresses from those simulations are not reported here.
\begin{figure}[t!]
	%\centering
	\includegraphics[scale=0.80]{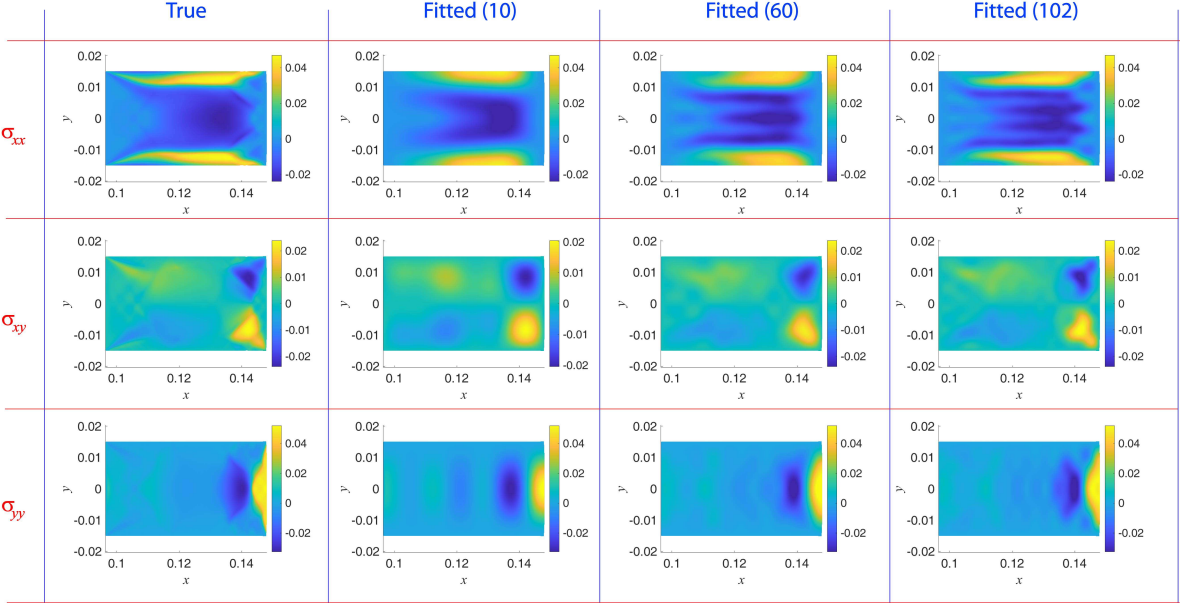}
	\caption{True and fitted (using 10, 60 and 102 eigenfunctions, respectively) stress components for metal forming example (in GPa).}
	\label{moreplots}
\end{figure}
\begin{figure}[t!]
	\centering
	\includegraphics[scale=0.85]{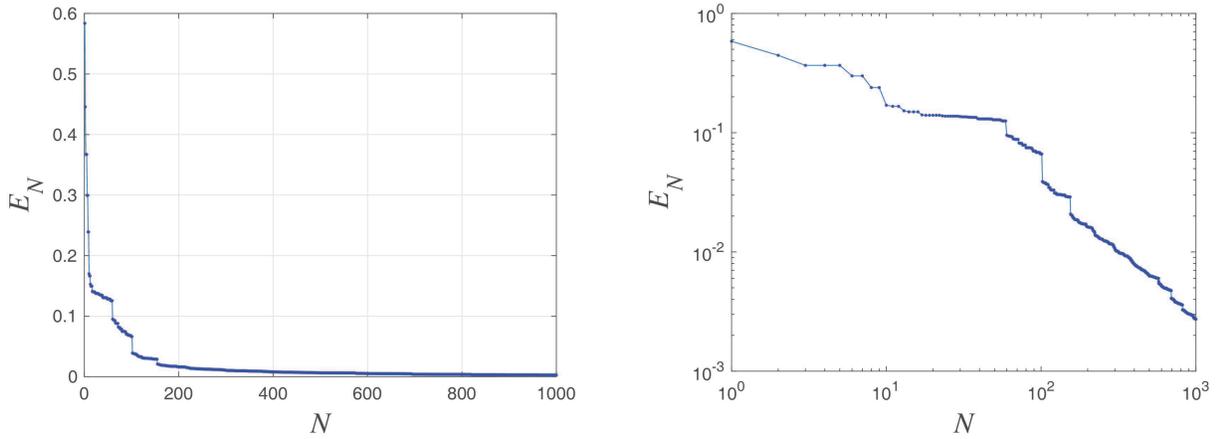}
	\caption{$E_N$ versus $N$, metal forming example. Left: linear scale; right: log-log scale.}
	\label{error}
\end{figure}
\begin{figure}[t!]
	\centering
	\includegraphics[scale=0.6]{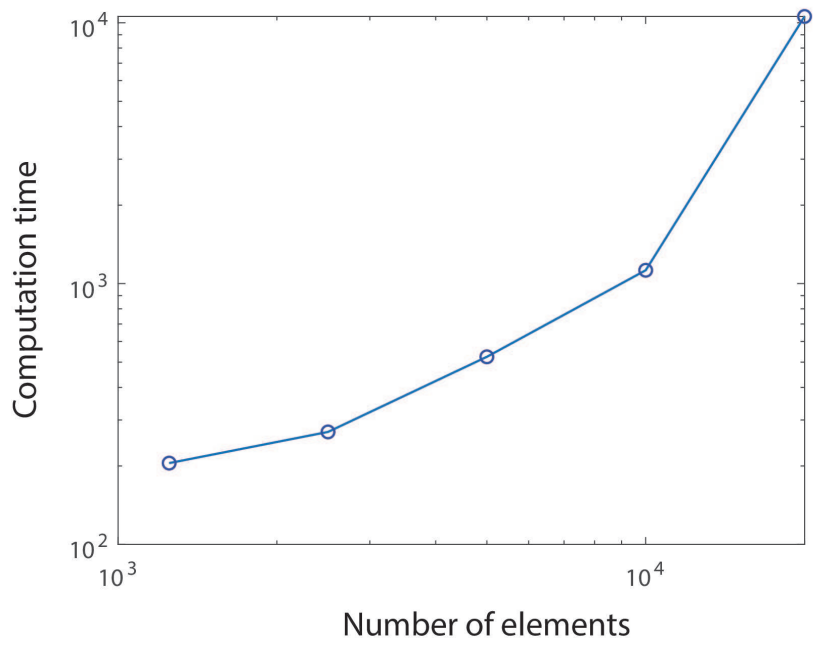}
	\caption{Time for computation of first one thousand eigenfunctions for different refinements of mesh used in metal forming simulation (log-log scale).}
	\label{comptimeroll}
\end{figure}
\end{subappendices}

\clearpage

\section{Stress fields used in section 7}  \label{expressions}
\underline{Example 1}
$$ \sigma_{rr}=-\frac{0.067}{r^2}+\frac{1.6}{r}-12.833+40 r-41.667 r^2, $$
$$\sigma_{r\theta}= -\frac{0.022}{r^2} +5.5 -40 r + 75  r^2 , $$
$$\sigma_{\theta\theta}= 3.667-40 r+100r^2.$$

\underline{Example 2}
$$\sigma_{rr}=\frac{- 0.321}{r} -\frac{- 4 \,{r}^{3}+8.563 \times 10^{-4} \,\sin \left(  200\,r \right) + 0.411\,\ln \left(  200\,r \right) r }{r^2}$$
$$-{\frac {9.408 \times 10^{-3}+ 7.611 \times 10^{-2}\,r\cos \left(  200\,r \right) }{{r}^{2}}}, $$
$$\sigma_{r\theta}= {\frac {{r}^{3}- 2.854\times 10^{-4} \,\sin \left( 200\,r \right) + 5.708 \times 10^{-2} \, r\cos \left( 200\,r \right) - 3.853 \times 10^{-2}\,r+ 7.840 \times 10^{-4}}{{r}^{2}}}, $$
$$\sigma_{\theta\theta}=- 3.805\,\sin \left( 200\,r \right) -\frac{1.284 \times 10^{-2}}{r} +r.$$

\underline{Example 3}
$$\sigma_{rr}(r)=-\frac{p_c}{\frac{r_c^2}{r_i^2}-1}\left(\frac{r_c^2}{r_i^2}-\frac{r_c^2}{r^2}\right) \hspace{3mm} \text{for} \hspace{3mm} r_i\leq r \leq r_c,$$
$$\sigma_{rr}(r)=-\frac{p_c}{\frac{r_o^2}{r_c^2}-1}\left(\frac{r_o^2}{r^2}-1\right) \hspace{3mm} \text{for} \hspace{3mm} r_c\leq r \leq r_o,$$  
$$\sigma_{\theta\theta}(r)=-\frac{p_c}{\frac{r_c^2}{r_i^2}-1}\left(\frac{r_c^2}{r_i^2}+\frac{r_c^2}{r^2}\right) \hspace{3mm} \text{for} \hspace{3mm} r_i\leq r \leq r_c,$$
$$\sigma_{\theta\theta}(r)=\frac{p_c}{\frac{r_o^2}{r_c^2}+1}\left(\frac{r_o^2}{r^2}+1\right) \hspace{3mm} \text{for} \hspace{3mm} r_c\leq r \leq r_o,$$ 
$$\sigma_{r\theta}(r)=0 \hspace{3mm} \text{for} \hspace{3mm} r_i\leq r \leq r_o,$$
where 
$$ p_c=\frac{E \delta}{r_c}\left\{\frac{1}{\frac{r_c^2+r_i^2}{\left(r_c^2-r_i^2\right)}-\nu} + \frac{1}{\frac{r_o^2+r_c^2}{ \left(r_o^2-r_c^2\right)}+\nu} \right\}.$$
We have used $r_i=0.1$, $r_c=0.2$, $r_o=0.3$, $\nu=0.3$, $E \delta=10^6$ in any consistent units.
\end{appendices}


\begin{thebibliography}{9}
	
	\bibitem{rayleigh} {Lord Rayleigh. 1877 {\em The theory of sound.} London, UK: Macmillan.}
	
	
	\bibitem{knops} {Knops RJ, Payne LE. 1971 {\em Uniqueness theorems in linear elasticity.} Berlin: Springer-Verlag.}
	
	\bibitem{schajer2013practical}  {Schajer GS. 2013 {\em Practical residual stress measurement methods.} Chichester, UK: John Wiley \& Sons.} 
	
	\bibitem{withers2001residual}  {Withers PJ, Bhadeshia HKDH. 2001  Residual stress. Part 1 - Measurement techniques. {\em Materials Science and Technology} {\bf 17} 355-365. (doi:{\color{blue}\href{https://doi.org/10.1179/026708301101509980}{10.1179/026708301101509980}})}
	
	\bibitem{withers2001residual2}  {Withers PJ, Bhadeshia HKDH. 2001  Residual stress. Part 2 - Nature and origins. {\em Materials Science and Technology} {\bf 17} 366-375. (doi:{\color{blue}\href{https://doi.org/10.1179/026708301101510087}{10.1179/026708301101510087}})}
	
	\bibitem{boley} {Boley BA, Weiner JH. 1960 {\em Theory of thermal stresses.} New York, NY: John Wiley and Sons.}
	
	\bibitem{eslami} {Eslami MR, Hetnarski RB, Ignaczak J, Noda N, Sumi N, Tanigawa Y. 2013 {\em Theory of elasticity and thermal stresses.} Dordrecht, The Netherlands: Springer.}
	
	\bibitem{eshelby1957determination} {Eshelby JD. 1957  The determination of the elastic field of an ellipsoidal inclusion, and related problems. {\em Proceedings of the Royal Society of London A} {\bf 241} 376-396. (doi:{\color{blue}\href{https://doi.org/10.1098/rspa.1957.0133}{10.1098/rspa.1957.0133}})}  
	
	\bibitem{eshelby1959elastic} {Eshelby JD. 1959  The elastic field outside an ellipsoidal inclusion. {\em Proceedings of the Royal Society of London A} {\bf 252} 561-569. (doi:{\color{blue}\href{https://doi.org/10.1098/rspa.1959.0173}{10.1098/rspa.1959.0173}})} 
	
	\bibitem{eshelby1957elastic} {Eshelby JD. 1958  The elastic model of lattice defects. {\em Annalen der Physik} {\bf 1} 116-121. (doi:{\color{blue}\href{https://doi.org/10.1002/andp.19574560113}{10.1002/andp.19574560113}})} 
	
	\bibitem{eshelby1966simple} {Eshelby JD. 1966  A simple derivation of the elastic field of an edge dislocation. {\em British Journal of Applied Physics} {\bf 17} 1131-1135. (doi:{\color{blue}\href{https://doi.org/10.1088/0508-3443/17/9/303}{10.1088/0508-3443/17/9/303}})} 
	
	\bibitem{kroener} {Kr\"{o}ner E. 1981  Continuum theory of defects. In {\em Les Houches, Session 35, 1980 - Physiques des Defaults \em (eds R. Balian, M Kl\'eman \& J. P. Poirer), pp. 215-315. New York, NY: North-Holland.}}
	
	\bibitem{mura} {Mura T. 1987 {\em Micromechanics of defects in solids.} Dordrecht, The Netherlands: Martinus Nijhoff Publishers.}
	
	\bibitem{goriely} {Goriely A. 2017 {\em The mathematics and mechanics of biological growth.} New York, NY: Springer.}
	
	\bibitem{zurlo} {Zurlo G, Truskinovsky L. 2017  Printing non-euclidean solids. {\em Physical Review Letters} {\bf 119} 048001. (doi:{\color{blue}\href{https://doi.org/10.1103/PhysRevLett.119.048001}{10.1103/PhysRevLett.119.048001}})} 
	
	\bibitem{swain} {Swain D, Gupta A. 2018  Biological growth in bodies with incoherent interfaces. {\em Proceedings of the Royal Society of London A} {\bf 474} 20170716. (doi:{\color{blue}\href{https://doi.org/10.1098/rspa.2017.0716}{10.1098/rspa.2017.0716}})}.
	
	\bibitem{epstein} {Epstein M. 2012  {\em The elements of continuum biomechanics.} Chichester, UK: John Wiley \& Sons.}
	
	\bibitem{hoger1986} {Hoger A. 1986  On the determination of residual stress in an elastic body. {\em Journal of Elasticity} {\bf 16} 303-324. (doi:{\color{blue}\href{https://doi.org/10.1007/BF00040818}{10.1007/BF00040818}})}
	
	\bibitem{hoger1985} {Hoger A. 1985  On the residual stress possible in an elastic body with material symmetry. {\em Archive for Rational Mechanics and Analysis} {\bf 88} 271-290. (doi:{\color{blue}\href{https://doi.org/10.1007/BF00752113}{10.1007/BF00752113}})} 
	
	\bibitem{shams} {Shams M, Destrade M, Ogden RW. 2011  Initial stresses in elastic solids: constitutive laws and acoustoelasticity. {\em Wave Motion} {\bf 48} 552-567. (doi:{\color{blue}\href{https://doi.org/10.1016/j.wavemoti.2011.04.004}{10.1016/j.wavemoti.2011.04.004}})} 
	
	\bibitem{gower} {Gower AL, Shearer T, Ciarletta P. 2017  A new restriction for initial stressed elastic solids. {\em The Quarterly Journal of Mechanics and Applied Mathematics} {\bf 70.4} 455-478. (doi:{\color{blue}\href{https://doi.org/10.1093/qjmam/hbx020}{10.1093/qjmam/hbx020}})}
	
	\bibitem{ISS} {Gower AL, Ciarletta P, Destrade M. 2015  Initial stress symmetry and its applications in elasticity. {\em Proceedings of the Royal Society A} {\bf 471} 20150448. (doi:{\color{blue}\href{https://doi.org/10.1098/rspa.2015.0448}{10.1098/rspa.2015.0448}})}
	
	\bibitem{shariff} {Shariff MHBM, Bustamante R, Merodio J. 2017  On the spectral analysis of residual stress in finite elasticity. {\em IMA Journal of Applied Mathematics} {\bf 82} 656-680. (doi:{\color{blue}\href{https://doi.org/10.1093/imamat/hxx007}{10.1093/imamat/hxx007}})}
	
	\bibitem{hartig} {Man CS. 1998  Hartig's law and linear elasticity with initial stress. {\em Inverse Problems} {\bf 14} 313-319. (doi:{\color{blue}\href{https://doi.org/10.1088/0266-5611/14/2/007}{10.1088/0266-5611/14/2/007}})}
	
	\bibitem{destrade} {Destrade M, Ogden RW. 2013  On stress dependent elastic moduli and wave speeds. {\em IMA Journal of Applied Mathematics} {\bf 78.5} 965-997. (doi:{\color{blue}\href{https://doi.org/10.1093/imamat/hxs003}{10.1093/imamat/hxs003}})}
	
	\bibitem{ciarletta} {Ciarletta P, Destrade M, Gower AL, Taffetani M. 2016  Morphology of residually stresses tubular tissues: beyond the elastic multiplicative decomposition. {\em Journal of the Mechanics and Physics of Solids} {\bf 90} 242-253. (doi:{\color{blue}\href{https://doi.org/10.1016/j.jmps.2016.02.020}{10.1016/j.jmps.2016.02.020}})} 
	
	\bibitem{prime1999residual} {Prime MB. 1999 Residual stress measurement by successive extension of a slot: the crack compliance method. {\em Applied Mechanics Reviews} {\bf 52} 75-96. (doi:{\color{blue}\href{https://doi.org/10.2172/481857}{10.2172/481857}})} 
	
	\bibitem{schajer2007residual} {Schajer GS, Prime MB. 2007 Residual stress solution extrapolation for the slitting method using equilibrium constraints. {\em Journal of Engineering Materials and Technology} {\bf 129} 226-232. (doi:{\color{blue}\href{https://doi.org/10.1115/1.2400281}{10.1115/1.2400281}})}
	
	\bibitem{akbari} {Akbari S, Taheri-Behrooz F, Shokrieh MM. 2013 Slitting measurement of residual hoop stresses through the wall-thickness of a filament wound composite ring. {\em Experimental Mechanics} {\bf 53} 1509-1518. (doi:{\color{blue}\href{https://doi.org/10.1007/s11340-013-9768-8}{10.1007/s11340-013-9768-8}})} 
	
	\bibitem{beghini} {Beghini M, Bertini L, Mori LF, Rosellini W. 2009 Genetic algorithm optimization of the hole-drilling method for non-uniform residual stress fields. {\em The Journal of Strain Analysis for Engineering Design} {\bf 44} 105-115. (doi:{\color{blue}\href{https://doi.org/10.1243/03093247JSA457}{10.1243/03093247JSA457}})} 
	
	\bibitem{ballard} {Ballard P, Constantinescu A. 1994  On the inversion of subsurface residual stresses from surface stress measurements. {\em Journal of the Mechanics and Physics of Solids} {\bf 42} 1767-1787. (doi:{\color{blue}\href{https://doi.org/10.1016/0022-5096(94)90071-X}{10.1016/0022-5096(94)90071-X}})} 
	
	\bibitem{robertson} {Robertson R. 1998  Determining residual stress from boundary measurements: a linearized approach. {\em Journal of Elasticity} {\bf 52} 63-73. (doi:{\color{blue}\href{https://doi.org/10.1023/A:1007551818084}{10.1023/A:1007551818084}})}
	
	\bibitem{acoustic} {Nedin R, Vatulyan A. 2013  Inverse problem of non-homogeneous residual stress identification in thin plates. {\em International Journal of Solids and Structures.} {\bf 50} 2107-2114. (doi:{\color{blue}\href{https://doi.org/10.1016/j.ijsolstr.2013.03.008}{10.1016/j.ijsolstr.2013.03.008}})} 
	
	\bibitem{gao} {Gao Z, Mura T. 1989  On the inversion of residual stresses from surface displacements. {\em Journal of Applied Mechanics.} {\bf 56(3)} 508-513. (doi:{\color{blue}\href{https://doi.org/10.1115/1.3176119}{10.1115/1.3176119}})} 
	
	\bibitem{schajerinv} {Schajer G, Prime MB. 2006 Use of inverse solutions for residual stress measurements. {\em Journal of Engineering Materials and Technology.} {\bf 128(3)} 375-382. (doi:{\color{blue}\href{https://doi.org/10.1115/1.2204952}{10.1115/1.2204952}})}
	
	\bibitem{faghidian} {Faghidian SA. 2014  A smoothed inverse eigenstrain method for reconstruction of the regularized residual fields. {\em International Journal of Solids and Structures.} {\bf 51} 4427-4434. (doi:{\color{blue}\href{https://doi.org/10.1016/j.ijsolstr.2014.09.012}{10.1016/j.ijsolstr.2014.09.012}})}
	
	\bibitem{pobedrja1978} {Pobedrja BE. 1978  Problems in terms of a stress tensor. {\em Doklady Akademii Nauk SSSR} {\bf 240} 564-567.} See {\color{blue}\href{http://mi.mathnet.ru/eng/dan41744}{http://mi.mathnet.ru/eng/dan41744}}.
	
	\bibitem{pobedrja1980} {Pobedrja BE. 1980  A new formulation of the problem in mechanics of a deformable solid body under stress. {\em Soviet Mathematics - Doklady} {\bf 22} 88-91.}
	
	\bibitem{li} {Li S, Gupta A, Markenscoff X. 2005  Conservation laws of linear elasticity in stress formulations. {\em Proceedings of the Royal Society A} {\bf 461} 99-116. (doi:{\color{blue}\href{https://doi.org/10.1098/rspa.2004.1347}{10.1098/rspa.2004.1347}})}
	
	\bibitem{markenscoff} {Markenscoff X, Gupta A. 2007  Configurational balance laws for incompatibility in stress space. {\em Proceedings of the Royal Society A} {\bf 463} 1379-1392. (doi:{\color{blue}\href{https://doi.org/10.1098/rspa.2007.1828} {10.1098/rspa.2007.1828}})} 
	
	\bibitem{gurtin} {Gurtin ME. 1972 The linear theory of elasticity. In {\em Handbuch der Physik} vol. VIa 2 (eds. S. Flugge \& C. Truesdell), pp. 1-295. Berlin, Germany: Springer.}
	
	\bibitem{temam} {Temam R. 1977 {\em Navier-Stokes equations - theory and numerical analysis.} New York, NY: North-Holland Publishing Company.}
	
	\bibitem{halmos} {Halmos PR. 1963  What does the spectral theorem say? {\em The American Mathematical Monthly} {\bf 70} 241-247.} See {\color{blue}\href{https://www.jstor.org/stable/2313117}{https://www.jstor.org/stable/2313117}}.
	
	\bibitem{Courant} {Courant R, Hilbert D. 1966 {\em Methods of mathematical physics, volume 1}. New York, NY: Interscience Publishers.}
	
	\bibitem{zienkiewicz} {Zienkiewicz OC. 1972 {\em Introductory lectures on the finite element method.} New York, NY: Springer.}
	
	\bibitem{gibbs1899fourier} {Hewitt E, Hewitt R E. 1979  The Gibbs-Wilbraham phenomenon: an episode in Fourier analysis. {\em Archive for History of Exact Sciences} {\bf 21} 129-160.} See {\color{blue}\href{https://www.jstor.org/stable/41133555}{https://www.jstor.org/stable/41133555}}.
	
	\bibitem{barber} {Barber {JR.} 1992 {{\em Elasticity.}} {Dordrecht, The Netherlands: Springer.}}
	
       \bibitem{Giovanni} {Giovanni L. 2009 {\em A first course in Sobolev spaces.} Providence, Rhode Island: American Mathematical Society.}

       \bibitem{brezis} {Brezis H. 2011 {\em Functional analysis, Sobolev spaces and partial differential equations.} New York, NY: Springer.}

       \bibitem{Bossavit} {Bossavit A.} 1998 {{\em Computational electromagnetism.} {Academic Press.}}

       \bibitem{ekland} {Ekeland I, T\'emam R. 1999 {\em Convex analysis and variational problems.} Philadelphia: Society for Industrial and Applied Mathematics.}	

       \bibitem{Bathe} {Bathe KJ.} 1996 {{\em Finite element procedures.} {New Jersey: Prentice-Hall.}
	}
	
	\bibitem{Bathe2} {Bathe KJ. 2001 {The inf-sup condition and its evaluation for mixed finite element methods.} {{\em Computers and Structures}} {\bf 79} 243-252. (doi:{\color{blue}\href{https://doi.org/10.1016/S0045-7949(00)00123-1} {10.1016/S0045-7949(00)00123-1}})} 
	
	\bibitem{Bathe3} {Bao W, Wang X, Bathe KJ. 2001 {On the inf-sup condition of mixed finite-element formulations for acoustic fluids.} {{\em Mathematical Models and Methods in Applied Sciences.}} {\bf 11(5)} 883-901. (doi:{\color{blue}\href{https://doi.org/10.1142/S0218202501001161} {10.1142/S0218202501001161}})} 	

     \bibitem{Falk} {Falk RS. 2008 Finite element methods for linear elasticity. In {\em Mixed finite elements, compatibility conditions, and applications}, pp. 159-194. New York, NY: Springer.}

\end{thebibliography}
\end{document}